\documentclass[12pt, leqno]{amsart}

\usepackage{stmaryrd,graphicx}
\usepackage{amssymb,mathrsfs,amsmath,amscd,amsthm,color}
\usepackage{float}
\usepackage{mathabx,mathdots,textcomp}
\usepackage[all,cmtip]{xy}
\DeclareMathAlphabet{\mathpzc}{OT1}{pzc}{m}{it}
\usepackage{amsfonts,latexsym,wasysym}
\usepackage[normalem]{ulem}
\usepackage[latin1]{inputenc}
\usepackage{cleveref}
\newcommand{\tripprox}{\setbox0\hbox{$\approx$}%
\mbox{\makebox[0pt][l]{\raisebox{0.48\ht0}{$\approx$}}$\approx$}}
\newcommand{\tripproxs}{\,\setbox0\hbox{$\approx$}%
\mbox{\makebox[0pt][l]{\raisebox{0.48\ht0}{$\approx$}}$\approx$}\,}

\newcommand{\light}{\lambda}
\newcommand{\im}{\text{Im}}
\newcommand{\diam}{\text{diam}}

\newcommand{\ord}{\preceq_{\ast}}

\newcommand{\ui}{[0,1]}

\newcommand{\den}{\simeq_{\bbr}}

\newcommand{\wt}{\widetilde}

\newcommand{\mco}{\mathcal{O}}

\newcommand{\scru}{\mathscr{U}}

\newcommand{\scrv}{\mathscr{V}}

\newcommand{\bbd}{\mathbb{D}}

\newcommand{\bbn}{\mathbb{N}}
\newcommand{\bbq}{\mathbb{Q}}
\newcommand{\bbr}{\mathbb{R}}

\newcommand{\bbz}{\mathbb{Z}}

\newcommand{\ov}{\overline}

\newcommand{\lc}{\text{lc}}
\newtheorem{theorem}{Theorem}[section]
\newtheorem{lemma}[theorem]{Lemma}
\newtheorem*{maintheorem*}{Main Theorem}
\newtheorem{proposition}[theorem]{Proposition}
\newtheorem{corollary}[theorem]{Corollary}
\theoremstyle{definition}\newtheorem{definition}[theorem]{Definition}

\newtheorem{remark}[theorem]{Remark}

\begin{document}
%
\title{On R-trees, homotopies, and covering maps}

\author[J. Brazas]{Jeremy Brazas}
\address{West Chester University\\ Department of Mathematics\\
West Chester, PA 19383, USA}
\email{jbrazas@wcupa.edu}

\author[G.R. Conner]{Gregory R. Conner}
\address{Brigham Young University\\ Department of Mathematics \\
Provo, UT 84602, USA}
\email{conner@math.byu.edu}

\author[P. Fabel]{Paul Fabel}
\address{Mississippi State University\\ Department of Mathematics and Statistics\\
Mississippi State, MS 39762, USA}
\email{fabel@math.msstate.edu}

\author[C. Kent]{Curtis Kent}
\address{Brigham Young University\\ Department of Mathematics \\
Provo, UT 84602, USA}
\email{curtkent@mathematics.byu.edu}

\date{\today}

\begin{abstract}

A map $p:E\to X$ has the \emph{unique path lifting} property if every path in $X$, after a choice of an initial point, lifts uniquely to a path in $E$.  We prove that if a group $G$ acts on an $\mathbb R$-tree $T$ such that the quotient map $p: T\to T/G$ has the unique path lifting property, then the quotient space $T/G$ does not contain a disc.  As a consequence, we show that every map of manifolds with the unique path lifting property is a covering map.  The proof requires a study of one-dimensional backtracking in paths.   We show the surprising  and counterintuitive result that the equivalence relation given by homotopies of paths rel.\ endpoints is generated by inserting and deleting one-dimensional backtracking.
\end{abstract}

\maketitle

\section{Introduction}
Actions of groups on $\mathbb R$-trees are a central tool in topology and geometric group theory.  Paulin constructed essential actions of  hyperbolic groups on $\mathbb R$-trees \cite{Paulin1988}.  The boundary of Culler and Vogtmann's Outer Space has a natural boundary consisting of actions of a free group on an $\mathbb R$-tree  \cite[Section 2]{BestvinaFeighn1994}.  Sela and Groves used isometric actions on $\mathbb R$-trees to study limit groups \cite{Groves2009, Sela2001}.  Rips developed tools to understand finitely generated groups acting freely on $\mathbb R$-trees (see \cite{BestvinaFeighn1995, GaboriauLevittPaulin1994}).

Our interest here is motivated by considering  non-finitely generated groups acting on $\mathbb R$-trees. Berestovskii  and Plaut in \cite{BP10rtrees}  prove that every length space $X$ is the quotient of an $\mathbb R$-tree $T$ obtained via an action by isometries of a locally free group $G$ on $T$.  In fact, their results show that the quotient map $p: T\to  T/G=X$ has the property that every rectifiable path in $X$ lifts uniquely to a path in $T$, after a choice of initial point.  They constructed the $\mathbb R$-tree $T$ by considering the set of rectifiable paths in $X$ without ``backtracking" and showed that these paths naturally formed an $\mathbb R$-tree.

We say that $p:E\to X$ has the \emph{unique path lifting property} (or is a \textit{UPL map}) if for every path $\alpha:\ui\to X$ and every $e\in p^{-1}\big(\big\{\alpha(0)\big\}\big)$, there is a unique path $\wt{\alpha}:\ui\to E$ with $\wt{\alpha}(0)=e$ and $p\circ \wt\alpha= \alpha$. The following is well-known: 

\begin{center}
    (\textasteriskcentered) If $p:E\to X$ is a covering map, then $p$ has the unique path lifting property.
\end{center}

If the quotient map $p:E\to X$ coming from a group action is a UPL map, one may understand the group action in terms of paths in the base space. Thus it is natural to ask when can a topological space $X$ be the orbit space of a group acting on an $\mathbb R$-tree such that the quotient map has the unique path lifting property.  We prove the following.

\begin{theorem}\label{orbit}
  If a group $G$ acts on an $\mathbb R$-tree $T$ such that the quotient map $p: T\to T/G$ has the unique path lifting property, then the quotient space $T/G$ does not contain a 2-dimensional Euclidean disc.
\end{theorem}

\noindent Theorem \ref{orbit} follows from the following lemma.

\begin{lemma}\label{lemdydaksproblem}
If $X$ is a first countable, locally path-connected, and simply connected space, $E$ is a path-connected space, and $p:E\to X$ is a UPL map, then $p$ is a homeomorphism.
\end{lemma}

This gives the following surprising corollary for manifolds.

\begin{corollary}\label{corcoveringmap}
If $p:E\to X$ is a UPL map where $X$ is first countable, locally path connected, and semilocally simply connected and $E$ is locally path connected, then $p$ is a covering map. In particular, every UPL map $p:M\to N$ of manifolds is a covering map.
\end{corollary}

\noindent  Lemma \ref{lemdydaksproblem} is a positive answer to Dydak's Unique Lifting Problem: Is every map with the unique path lifting  property from a path-connected space $E$ to the unit disc $\mathbb D^2$ in the Euclidean plane  a homeomorphism?  (See \cite[Problem 2.3]{Dydak} and \cite[ Problem 4.6]{BM}.) Additionally, Corollary \ref{corcoveringmap} is a converse to (\textasteriskcentered), which is surprisingly difficult to verify outside of one-dimensional cases.

The primary obstruction to establishing the converse to (\textasteriskcentered) is showing that the unique lifting of paths rel.\ basepoint implies the lifting of path homotopies. If one assumes that path-lifts vary continuously relative to their starting point, then standard techniques would apply, for example see  \cite[Section 12 of Chapter III]{Hu}. However, the weaker lifting hypothesis of UPL maps need not imply continuous lifting in general. In fact, the aforementioned results of Berestovskii-Plaut may be considered as evidence that the converse of (\textasteriskcentered) should not hold even when $X=\bbd^2$.  Since arbitrary paths in $\bbd^2$ can be approximated by paths in one-dimensional subspaces and since $\bbr$-trees are closed under various limiting constructions, limiting methods seem a promising way to produce a counterexample extending that for rectifiable paths. 

Our principal tool is understanding ``one-dimensional backtracking." We will make backtracking formal by considering maps into $\mathbb R$-trees.  Suppose the non-constant path $\alpha:\ui\to X$ in a metric space factors through an $\bbr$-tree $T$ as $\alpha=q \circ p $ for a path $p:\ui\to T$ and a map $q:T\to X$. Let $r:\ui\to T$ parameterize the unique geodesic in $T$ from $p(0)$ to $p(1)$. Then we say that the path $\beta=q \circ r$ is obtained from $\alpha$ by \emph{geodesic $\bbr$-tree reduction}. Certainly, if $\beta$ is a geodesic $\bbr$-tree reduction of $\alpha$, then $\alpha$ and $\beta$ are homotopic. The surprising result and main idea to prove Theorem \ref{orbit} and Corollary \ref{corcoveringmap} is that any two homotopic paths are both $\bbr$-tree reductions of a single common path. Hence, one may delete ``one-dimensional backtracking" from this common path to obtain either of the two homotopic paths.

\begin{theorem}\label{theoremmain}
If $X$ is a topological space and $\alpha,\beta:\ui\to X$ are homotopic paths, then there exists a path $\gamma:\ui\to X$ such that $\alpha$ and $\beta$ are both geodesic $\bbr$-tree reductions of $\gamma$.
\end{theorem}

Typically, the path $\gamma$ will be very complicated, as it (1) will be space-filling in the image of some chosen homotopy between $\alpha$ and $\beta$ and (2) must pass through all of the points of both $\alpha$ and $\beta$ in the same order that each of these paths does. Despite its complexity, the construction of $\gamma$ from the pair $\alpha,\beta$ is entirely explicit. Our Main Theorem implies that the equivalence relation on paths generated by geodesic $\bbr$-tree reduction coincides precisely with path-homotopy (Corollary \ref{corcoincideswithhomotopy}).

In \cite[Corollary 2.4]{HKK}, it is shown that the following notion of ``thin homotopy" is an equivalence relation on the set of paths in a Hausdorff space: \textit{two paths are equivalent if they form a loop that factors through a simplicial tree}. The key idea needed to verify transitivity is a technical result (\cite[Proposition 5.5]{HKK}) ensuring that certain pushouts of simplicial trees are simplicial trees. Theorem \ref{theoremmain} implies that if one replaces ``simplicial tree" with ``$\bbr$-tree" in the definition of thin homotopy, then the analogous relation is \textit{not} transitive (Corollary \ref{corollarynottransitive}). This occurs precisely because there exist extreme situations, namely those modeled by our main construction, where analogous pushouts of $\bbr$-trees are topological discs.

\subsection*{Outline of paper}
Section \ref{sec: equivalences} establishes notation for four equivalence relations on the path space that we will require throughout the paper and some elementary properties of these relations.  Section \ref{sec: Canor paths} defines Cantor paths and staggered paths, which we will use in our construction of the path $\gamma$ from Theorem \ref{theoremmain} to control the distance between approximations of $\gamma$.
In Section \ref{sec: inserting inverse pairs}, we define CIP-loops and LIP-loops and state Lemma \ref{gammaextensionlemma} and Lemma \ref{alternatinglemma}.  These two lemmas allow us to conclude that whenever a given path is modified by a recursive process of inserting out-and-back paths along parts of the domain where it is already constant, the uniform limit is $\bbr$-tree homotopic to the original path.  In Section \ref{sec: proof of main theorem}  we prove the technical result, Lemma \ref{fabelsconstruction}, that is the inductive step in the construction of  $\gamma$ for Theorem \ref{theoremmain} and complete the proof of Theorem \ref{theoremmain}.  Section \ref{sec: Dydak's problem} contains the proofs of Theorem \ref{orbit}, Lemma \ref{lemdydaksproblem}, and Corollary \ref{corcoveringmap}.  Section \ref{appendix} is the appendix and contains the proof of Lemma \ref{gammaextensionlemma}.

\section{$\bbr$-tree factorization of paths and loops}\label{sec: equivalences}

For paths $\alpha$ and $\beta$ in a space $X$, $\alpha\beta$ will denote path concatenation and $\ov{\alpha}$ will denote the reverse path. If $(X,d)$ is a metric space, we let $\rho(\alpha,\beta)=\sup\{d(\alpha(t),\beta(t))\mid t\in\ui\}$ denote the sup-metric on the space of paths from $\ui$ to $X$ and recall that when $(X,d)$ is complete, then so is its path space \cite[4.3.13]{Engelking}. We also require notation for a variety of relations on paths.

\begin{definition}\label{deffrechet}
Let $\alpha:[a,b]\to X$, $\beta:[c,d]\to X$ be paths.
\begin{enumerate}
\item We say $\alpha$ is \textit{equivalent} to $\beta$ and write $\alpha\equiv\beta$ if $\alpha=\beta\circ h$ for some increasing homeomorphism $h: [a,b]\to [c,d]$. If $h$ is linear, we may say that $\alpha$ is a \textit{linear reparameterization of} $\beta$.
\item We say $\alpha$ and $\beta$ are \textit{Fr\'echet equivalent} and write $\alpha\tripproxs\beta$ if $\alpha\circ f=\beta\circ g$ for non-decreasing continuous surjections $f:\ui\to [a,b]$ and $g:\ui\to [c,d]$.
\end{enumerate}
\end{definition}
Both $\equiv$ and $\tripprox$ are equivalence relations finer than the path-homotopy relation $\simeq$. See \cite{CF} for a proof of the transitivity of $\tripprox$.  Recall that an \textit{$\bbr$-tree} is a uniquely arcwise-connected, and locally arcwise connected geodesic metric space \cite{bestvinartrees,MORtree} and a \emph{Peano continuum} is a compact, connected, locally path-connected, metric space.      A  \textit{dendrite} is Peano continuum that does not contain a simple closed curve \cite[Definition 10.1]{Nadler}. It is then an exercise to see that a dendrite is uniquely arcwise-connected and locally arcwise connected.  Mayer and Oversteegen show that any uniquely arcwise connected and locally arcwise connected metric space admits a geodesic metric, see \cite[Theorem 5.1]{MORtree}.  Thus every dendrite admits a metric that makes it an $\bbr$-tree.

\begin{definition}\label{defden}
Let $\alpha,\beta:\ui\to X$ be paths in a topological space $X$.
\begin{enumerate}
\item We say that $\beta$ is a \textit{geodesic $\bbr$-tree reduction of} $\alpha$, and we write $\alpha\geq_{\bbr}\beta$, if there is an $\bbr$-tree $T$, a path $p:\ui\to T$, an injective path $r:\ui\to T$ with $p(i)=r(i)$ for $i\in\{0,1\}$, and a map $f:T\to X$ such that $f\circ p=\alpha$ and $f\circ r=\beta$.
\item We say that $\alpha$ and $\beta$ are \textit{$\bbr$-tree homotopic}, and we write $\alpha\den\beta$, if $\alpha\ov{\beta}$ is a loop that factors through an $\bbr$-tree, that is, if $\alpha(i)=\beta(i)$ for $i\in\{0,1\}$ and if there exists an $\bbr$-tree $T$, a loop $\ell:\ui\to T$ and a map $f:T\to X$ such that $f\circ\ell=\alpha\ov{\beta}$.
\end{enumerate}
\end{definition}

Note that \[\alpha\equiv\beta\text{ }\Rightarrow\text{ }\alpha\geq_{\bbr}\beta\text{ }\Rightarrow\text{ }\alpha\den\beta\text{ }\Rightarrow\text{ }\alpha\simeq \beta\] where the last implication holds since $\bbr$-trees are contractible. Geodesic $\bbr$-tree reduction is our formalization of removing backtracking, as mentioned in the introduction and abstract. Also, if $\alpha\den\beta$ where $\beta$ is injective, then $\alpha\geq_{\bbr}\beta$. Certainly, $\den$ is reflexive and symmetric. However, the transitivity of $\den$ does not hold in general, see Corollary \ref{corollarynottransitive}. In a given space, transitivity of $\den$ is equivalent to whether or not each path admits a unique (up to $\equiv$) ``maximally reduced" geodesic $\bbr$-tree reduction. One must be wary of this temptation as our main result implies that there exist some paths in the plane which fail to have unique maximally reduced representatives. We note some key properties of $\den$ and $\geq_{\bbr}$ that \textit{do} hold.

\begin{remark}
If $T_1$ and $T_2$ are $\bbr$-trees, $f_1:\ui\to T_1$ is an injective path, and $f_2:\ui\to T_2$ is any path, then the pushout of $f_1$ and $f_2$ is an $\bbr$-tree. Indeed, this pushout is obtained by attaching the closure of each connected component of $T_1\backslash f_1(\ui)$ to $T_2$ at a point along the image of $f_2$. This fact can be used to show the following.
\begin{enumerate}
\item $\geq_{\bbr}$ is transitive and is antisymmetric up to equivalence, that is, $\alpha\geq_{\bbr}\beta$ and $\beta \geq_{\bbr}\alpha$ $\Rightarrow$ $\alpha\equiv \beta$. Hence, $\geq_{\bbr}$ induces a partial order on path-equivalence classes.
\item $\alpha\den\beta$ if and only if $\alpha$ and $\beta$ share a common geodesic $\bbr$-tree reduction, that is, if and only if there exists $\delta$ with $\alpha\geq_{\bbr}\delta$ and $\beta\geq_{\bbr}\delta$.
\end{enumerate}
\end{remark}

\begin{lemma}\label{denproperties}
The following properties of $\bbr$-tree homotopy hold.
\begin{enumerate}
\item If $\alpha,\beta:\ui\to X$ are paths such that $\alpha\den\beta$ (resp. $\alpha\geq_{\bbr}\beta$) and $f:X\to Y$ is a map, then $f\circ\alpha\den f\circ \beta$ (resp. $f\circ\alpha\geq_{\bbr}f\circ \beta$).
\item If $\alpha_1\den\beta_1$, $\alpha_2\den\beta_2$, and $\alpha_1(1)=\alpha_2(0)$, then $\alpha_1\alpha_2\den \beta_1\beta_2$.

\hspace{-5em}  If $X$ is Hausdorff, then the following also hold.

\item If $\alpha,\beta:\ui\to X$ are paths and $\alpha\circ g_1\den \beta\circ g_2$ for non-decreasing continuous surjections $g_1,g_2:\ui\to\ui$, then $\alpha\den\beta$.
\item If $\alpha_1,\alpha_2:\ui\to X$ are paths such that $\alpha_1\den\alpha_2$ and $\beta_1,\beta_2:\ui\to X$ are paths such that $\alpha_1\tripproxs\beta_1$ and $\alpha_2\tripproxs\beta_2$, then $\beta_1\den\beta_2$.
\end{enumerate}
\end{lemma}

\begin{proof}
(1) is clear. (2) holds since the one-point union of two $\bbr$-trees is an $\bbr$-tree. (3) Since $\alpha\circ g_1\den \beta\circ g_2$, there exists an $\bbr$-tree $T$, map $F:T\to X$, and loop $\ell:\ui \to T$ such that $F\circ\ell=(\alpha\circ g_1) (\ov{\beta\circ g_2})$. We may replace $T$ with the dendrite $\ell(\ui)$ and assume $\ell$ is surjective. Since $X$ is assumed to be Hausdorff, $F(T)$ is a compact metric space \cite[8.17]{Nadler} and we may replace $X$ with $F(T)$. Define a non-decreasing surjection $g:\ui\to \ui$ by
\[g(t)=\begin{cases}
\frac{1}{2}g_1(2t), & \text{ if }t\in[0,\frac{1}{2}],\\
1-\frac{1}{2}g_2(2-2t), & \text{ if }t\in[\frac{1}{2},1]
\end{cases}\]
so that $F\circ \ell=(\alpha\ov{\beta})\circ g$. Applying the Monotone-Light Factorization Theorem, see \cite[Theorem 1]{Eilenberg}  or  \cite[Theorem 2.3]{Whyburn}, to the map $F:T\to X$ of compact metric spaces, we have $F=F'\circ \pi$ for a monotone map $\pi:T\to T'$ and a light map $F':T'\to X$. Since $\pi$ is a monotone quotient map on a dendrite, $T'$ is a dendrite \cite[Exercise 10.52]{Nadler}. Since $g$ is monotone, $F'$ is light, and $F'\circ \pi\circ \ell=(\alpha\ov{\beta})\circ g$, the loop $\pi\circ \ell:\ui\to T'$ is constant on the fibers of $g$. Thus, there is a unique loop $\mu:\ui\to T'$ such that $\pi\circ \ell=\mu\circ g$.
\[\xymatrix{
\ui \ar[dd]_-{g} \ar[rr]^-{\ell} && T \ar[dd]^-{F} \ar[dl]_-{\pi} \\
& T' \ar[dr]_-{F'}\\
\ui \ar[ur]^-{\mu} \ar[rr]_-{\alpha\ov{\beta}} & & X
}\]
Since $F'\circ \mu\circ g=(\alpha\ov{\beta})\circ g$ where $g$ is surjective, we have $F'\circ \mu=\alpha\ov{\beta}$, proving $\alpha\den\beta$.

(4) Write $\alpha_1\circ f_1=\beta_1\circ g_1$ and $\alpha_2\circ f_2=\beta_2\circ g_2$ for non-decreasing continuous surjections $f_1,f_2,g_1,g_2:\ui\to\ui$. Since $\alpha_1\den\alpha_2$, there exists an $\bbr$-tree $T$, a map $F:T\to X$, and loop $\ell:\ui\to T$ such that $F\circ \ell=\alpha_1\ov{\alpha_{2}}$. Write $\ell=\mu_1\ov{\mu_{2}}$ so that $F\circ\mu_i=\alpha_i$, $i\in\{1,2\}$. If $\ell '=(\mu_1\circ f_1)(\ov{\mu_2\circ f_2})$, then the factorization \[F\circ \ell '=(\alpha_1\circ f_1)(\ov{\alpha_2\circ f_2})=(\beta_1\circ g_1)(\ov{\beta_2\circ g_2})\] shows that $\beta_1\circ g_1\den\beta_2\circ g_2$. Now, (3) implies $\beta_1\den\beta_2$.
\end{proof}

\section{Cantor paths and their staggering}\label{sec: Canor paths}
\begin{definition}
 Recall that a \emph{Cantor set} is any non-empty, compact, perfect, and totally disconnected metric space and any such space is homeomorphic to the standard ternary Cantor set.  An open set $U\subseteq (0,1)$ is a \textit{Cantor complement} if $\ui\backslash U$ is homeomorphic to a Cantor set.
\end{definition}

Note that an open set $U\subseteq (0,1)$ is a Cantor complement if and only if $U$ is dense in $\ui$ and the set of connected components of $U$ have the order type of $\bbq$. It is necessarily the case that the connected components of a Cantor complement have pairwise-disjoint closures. Given a path $\alpha:\ui\to X$, let $\lc(\alpha)$ denote the set of connected components of the open set $\mco(\alpha)=\{t\in\ui\mid \alpha\text{ is locally constant at }t\}$. Note that distinct elements of $\lc(\alpha)$ have pairwise-disjoint closures and that $\alpha$ is a light map if and only if $\mco(\alpha)=\emptyset$.

\begin{definition}
A non-constant path $\alpha:\ui\to X$ is a \textit{Cantor path} if $\mco(\alpha)$ is a Cantor complement.
\end{definition}

The standard ternary Cantor map $\tau:\ui\to\ui$ collapses the closure of each component of the complement of the ternary Cantor set to a point. Thus $\tau$ is a non-decreasing, surjective Cantor path. If $\alpha$ is a light path, then $\alpha\circ \tau$ is a Cantor path. A key concept in the proof of our main theorem is the following.

\begin{definition}\label{staggereddef}
Let $U_1,U_2$ be proper open subsets of $(0,1)$ such that for each $i\in\{1,2\}$, $\inf(U_i)=0$, $\sup(U_i)=1$, and such that the connected components of $U_i$ have pairwise disjoint closures. We say the sets $U_1$ and $U_2$ are \textit{staggered} if $U_1\cup U_2=(0,1)$, or equivalently, if $\partial U_1\cap \partial U_2=\{0,1\}$.
\end{definition}

The next lemma, which has a straightforward proof, allows us to select connected components from staggered Cantor complements $U_1$ and $U_2$ so that the resulting collections have the order type of $\bbz$ while still having staggered unions.

\begin{lemma}\label{choosezsetlemma}
Let $U_1$ and $U_2$ be staggered Cantor complements. Then for $i\in\{1,2\}$ there exists a set of connected components $\scru_i$ of $U_i$ such that, $\scru_i$ has the order type of $\bbz$, $\inf(\bigcup\scru_i)=0$, $\sup(\bigcup\scru_i)=1$, and such that $\bigcup\scru_1$ and $\bigcup\scru_2$ are staggered.
\end{lemma}

\begin{proof}
    Let $\scru_1$ be the components of $U_1$ that are not entirely contained in some component of $U_2$.  Let $\scru_2$ be the components of $U_2$ which are not entirely contained in some element of $\scru_1$.  Note that $\scru_1\cup\scru_2$ still covers $(0,1)$. Now suppose the set of endpoints of $\scru_1$ has an accumulation point, $p$,  that is neither $0$ nor $1$. Since the components of $\scru_1$ are disjoint and open, $p$ is not contained in any element of $\scru_1$. Since  $\scru_1\cup\scru_2$ covers  $(0,1)$, $p$ is contained in some  $O\in\scru_2$. By construction each element of $\scru_2$ can contain at most two endpoints of $\scru_1$, which contradicts that $p$ is an accumulation point of the endpoints of $\scru_1$.

    Note that since both $\scru_1$ and $\scru_2$ are nonempty and each component  separates $\bbr$, it must be that the set of components, and thus the set of endpoints, of $\scru_1$ is infinite, does not contain $0$ or $1$, but has both as limit points.

    Therefore the set of endpoints of $\scru_1$ is discrete and (being a subset of $\bbr$) linear and has limit points at $0$ and $1$ it must be a bi-infinite sequence and thus have order type $\bbz$.  A similar argument holds for $\scru_2$.
\end{proof}

\begin{lemma}\label{lemmaccsetmove}
Let $U$ be a Cantor complement. For every $\epsilon>0$, there exists an increasing homeomorphism $f:\ui\to \ui$ such that $\rho(f,id_{\ui})<\epsilon$ and such that $U$ and $f(U)$ are staggered.
\end{lemma}

\begin{proof}
Let $\epsilon>0$ and $\scru$ be the set of connected components of $U$ with the induced ordering. Since $U$ is a Cantor complement, $\scru$ is densely ordered and we may find a bi-infinite sequence $I_n=(a_n,b_n)\in\scru$, $n\in\bbz$ such that $\inf(\bigcup_{n\in\bbn}I_n)=0$, $\sup(\bigcup_{n\in\bbn}I_n)=1$, and such that for all $n\in\bbz$, the segment $[b_n,a_{n+1}]$ has diameter less than $\frac{\epsilon}{3}$. For each $n\in\bbz$, find $J_n=(c_n,d_n)\in \scru$ with $J\subseteq (b_n,a_{n+1})$. We will define $f$ by first setting its value on the closures of the intervals in the bi-infinite sequence \[\cdots <I_{-2}<J_{-2}<I_{-1}<J_{-1}<I_0<J_0<I_1<J_1<I_2<J_2<\cdots.\]

For each $n\in\bbz$, choose a subdivision of $[a_n,b_n]$ as follows: find $a_n<D_{n-1}<A_{n}<B_{n}<C_{n}<b_n$ where the outer four segments $[a_n,D_{n-1}]$, $[D_{n-1},A_n]$, $[B_n,C_n]$, and $[C_n,b_n]$ all have diameter less that $\frac{\epsilon}{3}$. Note that if $K_n=(A_n,B_n)$ and $L_n=(C_n,D_n)$, then the bi-infinite sequence \[\cdots <K_{-2}<L_{-2}<K_{-1}<L_{-1}<K_0<L_0<K_1<L_1<K_2<L_2<\cdots\] limits to $0$ on the left and $1$ on the right. Moreover, $K_n\subseteq I_n$ and $J_n\subseteq L_n$ for all $n\in\bbz$. We define $f$ so that it maps $[a_n,b_n]$ to $[A_n,B_n]$ and $[c_n,d_n]$ to $[C_n,D_n]$ respectively by increasing linear maps. Moreover, since the sets $I_n,J_n,K_n,L_n$ limit to $0$ as $n\to-\infty$ and to $1$ as $n\to\infty$, this definition extends uniquely to an increasing homeomorphism $f:\ui\to\ui$, which is piecewise-linear on $[a_{-n},b_{n}]$ for all $n\geq 1$ (see Figure \ref{figperturb}).

\begin{figure}[H]
\centering \includegraphics[height=2.3in]{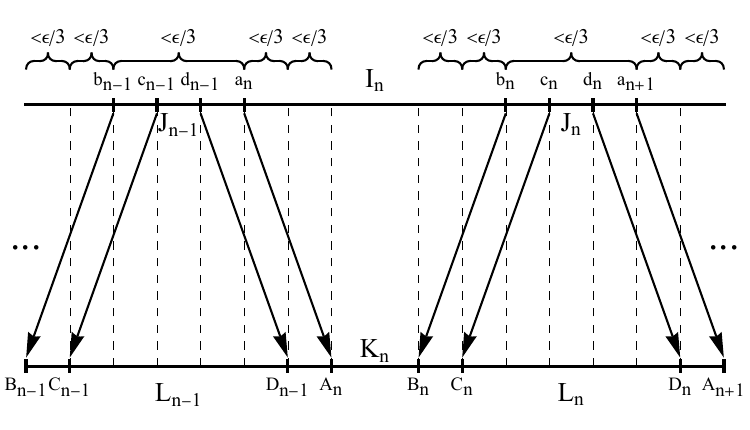}
\caption{\label{figperturb}An illustration of part of the homeomorphism $f:\ui\to\ui$ which maps $I_n\to K_n$ and $J_n\to L_n$ in a piecewise-linear fashion.}
\end{figure}

Our choices of the sizes of the intervals $[b_n,a_{n+1}]$ and subdivisions ensures that $\rho(f,id_{\ui})<\epsilon$. Additionally, if $t\in (0,1)\backslash U$, then $t\in [b_n,c_n]\cup [d_n,a_{n+1}]$ for some $n\in \bbz$. Since $[b_n,c_n]\cup [d_n,a_{n+1}]\subseteq L_n=f(J_n)$ where $J_n\in\scru$, we have $t\in f(U)$. We conclude that $U\cup f(U)=(0,1)$, i.e. $U$ and $f(U)$ are staggered.
\end{proof}

\begin{definition}
We say two Cantor paths $\alpha,\beta:\ui\to X$ are \textit{staggered} if $\mco(\alpha)$ and $\mco(\beta)$ are staggered.
\end{definition}

Our final lemma of the section allows us to take two non-staggered Cantor paths and perturb one of them so that the resulting pair is staggered.

\begin{lemma}\label{staggeredlemma}
Given any two Cantor paths $\alpha,\beta:\ui\to X$ and $\epsilon,\delta>0$, there exists an increasing homeomorphism $f:\ui\to \ui$ such that $\rho(f,id_{\ui})<\delta$, $\rho(\alpha,\alpha\circ f)<\epsilon$, and such that $\alpha\circ f$ and $\beta$ are staggered.
\end{lemma}

\begin{proof}
Let $U_1=\mco(\alpha)$ and $U_2=\mco(\beta)$ and note that $V=U_1\cap U_2$ is also a Cantor complement. Find $0<\delta '<\delta$ such that $|s-t|<\delta'$ implies $|\alpha(s)-\alpha(t)|<\frac{\epsilon}{2}$. By Lemma \ref{lemmaccsetmove}, we may find an increasing homeomorphism $f:\ui\to \ui$ such that $|f(t)-t|<\delta'$ for all $t\in\ui$ and $V\cup f(V)=(0,1)$. Hence, $|\alpha(f(t))-\alpha(t)|<\frac{\epsilon}{2}$ for all $t\in\ui$, which gives $\rho(f,id_{\ui})\leq \delta' <\delta$ and $\rho(\alpha,\alpha\circ f)<\epsilon$. Also, note that $V\cup f^{-1}(V)=(0,1)$ and $\mco(\alpha\circ f)=f^{-1}(U_1)$. Thus
\[(0,1)=V\cup f^{-1}(V)=(U_1\cap U_2)\cup (f^{-1}(U_1)\cap f^{-1}(U_2))\subseteq U_2\cup\mco(\alpha\circ f),\]showing that $\alpha\circ f$ and $\beta$ are staggered.
\end{proof}

\section{Inserting inverse pairs into Cantor paths}\label{sec: inserting inverse pairs}

If $\scru$ is a collection of open intervals in $\ui$ with disjoint closures, then a \textit{$\scru$-collapsing map} is a non-decreasing, continuous surjection $\mu:\ui\to\ui$, which is constant on the closure of each $J\in\scru$ and which is bijective on  $\ui\backslash \bigcup_{J\in\scru}\ov{J}$. Such maps may be constructed canonically using dyadic rational outputs and by enumerating $\scru$ by non-increasing length and the ordering in $\ui$ (lexicographically). If a path $\alpha:\ui\to X$ is not light and $k_{\alpha}:\ui\to\ui$ is a $\lc(\alpha)$-collapsing map, then there is a unique light path $\alpha
^{\light}:\ui\to X$ such that $\alpha^{\light}\circ k_{\alpha}=\alpha$.

\begin{definition}
We call a given loop $\ell:\ui\to X$ an \textit{inverse-pair loop} if there exists a path $\alpha:\ui\to X$ such that $\ell\equiv \alpha\ov{\alpha}$. More specifically,
\begin{enumerate}
\item if $\alpha$ is a Cantor path or has the form $\alpha\equiv \alpha_1 \alpha_2$ for Cantor path $\alpha_1$ and constant path $\alpha_2$, then, we call $\ell$ a \textit{Cantor-inverse-pair loop} or \textit{CIP-loop} (see Figure \ref{figcip}).
\item if $\alpha$ is light, we call $\ell$ a \textit{light-inverse-pair loop} or \textit{LIP-loop}.
\end{enumerate}
\end{definition}

\begin{figure}[H]
\centering \includegraphics[height=1.5in]{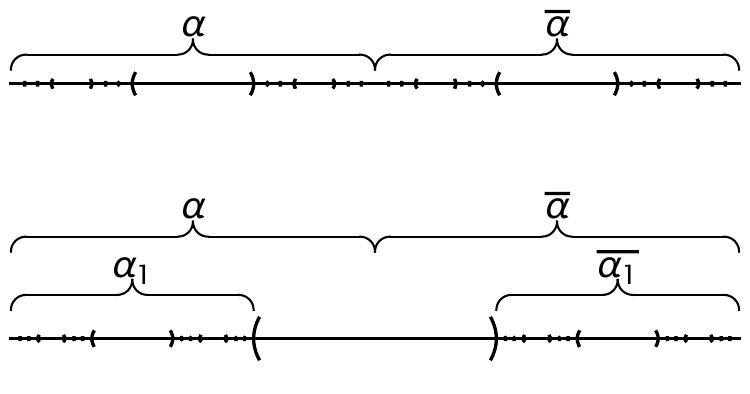}
\caption{\label{figcip}A CIP-loop is an inverse pair loop which either has the form of an inverse pair $\alpha\ov{\alpha}$ of Cantor paths (above) or of the form $\alpha_1 \beta  \ov{\alpha_1}$ for a Cantor path $\alpha_1$ and constant path $\beta$ (below).}
\end{figure}

\begin{remark}\label{liploopremark}
If $\alpha:\ui\to X$ is a CIP-loop, then $\alpha$ is a Cantor path. Moreover, if $\alpha=\alpha^{\light}\circ k_{\alpha}$ for $\lc(\alpha)$-collapsing map $k_{\alpha}:\ui\to\ui$ and light path $\alpha^{\light}:\ui\to X$, then $\alpha^{\light}$ is a LIP-loop.
\end{remark}

\begin{definition}\label{orderdef}
Let $\alpha,\beta:\ui\to X$ be Cantor paths and $\scru\subseteq\lc(\alpha)$. We say that $\beta$ is a \textit{$\scru$-extension} of $\alpha$ if
\begin{enumerate}
\item $\beta(t)=\alpha(t)$ for all $t\in\ui\backslash \bigcup\scru$,
\item for each $J\in \scru$, $\beta|_{\ov{J}}$ is a CIP-loop.
\end{enumerate}
If $\beta$ is a $\scru$-extension of $\alpha$ for some subset $\scru\subseteq\lc(\alpha)$, we write $\alpha\ord\beta$.
\end{definition}

\begin{remark}\label{extremark}
If $\alpha\ord\beta$, then $\mco(\beta)\subseteq \mco(\alpha)$. In particular, if $J\in \lc(\beta)$, then either $J\in \lc(\alpha)\backslash\scru$ or $J\subseteq K$ for some $K\in\scru$. Moreover, the loop $\alpha\ov{\beta}$ factors through a dendrite in a highly structured way. In particular, there is a dendrite $D(\beta,\alpha)$ constructed by starting with a ``base arc" $B$ and attaching an arc to $B$ for each element of $\scru$. We have $F\circ (a\ov{b})=\alpha\ov{\beta}$ where $a:\ui\to D(\beta,\alpha)$ is a monotone map onto $B$ and where $b:\ui\to D(\beta,\alpha)$ maps $\ui\backslash\bigcup\scru$ into $B$ and the closure of each element of $\scru$ onto the corresponding attached arc in $D(\beta,\alpha)$ by an inverse-pair loop. While $\alpha\den\beta$ holds, the relation $\alpha\leq_{\bbr}\beta$ only holds in the trivial case where $\scru=\emptyset$ (see the Section \ref{appendix} for more details).
\end{remark}

\begin{remark}
The relation $\ord$ on the set of paths in a space $X$ is certainly reflexive and it is straightforward to check that $\ord$ antisymmetric. However, $\ord$ is \textit{not} transitive. Rather, $\ord$ is a very fine relation that generates a partial order relation, which is strictly finer than $\leq_{\bbr}$.
\end{remark}

If we have a sequence $\gamma_1\ord\gamma_{2}\ord \gamma_{3}\ord\cdots$ in a space $X$ where, as one proceeds through the sequence, the added out-and-back loops have very quickly shrinking diameters, then $\{\gamma_{n}\}_{n\in\bbn}$ should converge uniformly to a path $\gamma$. Moreover, $\gamma_1\ov{\gamma_{2}}$ factors through a dendrite, call it $D_1$, as described in Remark \ref{extremark}. Since $\gamma_{n+1}$ agrees with $\gamma_{n}$ except on portions on which $\gamma_{n+1}$ is a CIP-loop, we may recursively construct a dendrite $D_n$ by attaching arcs to $D_{n-1}$ so that $\gamma_j\ov{\gamma_{n+1}}$ factors through $D_{n}$ for each $j\in\{1,2,\dots,n\}$. In the limit, we find that there is a uniquely determined limit dendrite $D_{\infty}=\varprojlim_{n}D_n$ that $\gamma_{n}\ov{\gamma}$ factors through for all $n$. Hence, $\gamma\den\gamma_n$ for all $n\in\bbn$. The next lemma establishes this conclusion assuming the existence of the limit $\gamma$. As one can see from the above proof sketch, this result is fairly intuitive. However, a detailed proof requires careful bookkeeping of parameterizations of paths in inverse limits. We omit the details here and include them in the Appendix, Section \ref{appendix}.

\begin{lemma}\label{gammaextensionlemma}
If $\{\gamma_n\}_{n\in\bbn}$ is a sequence of Cantor paths in a metric space $(X,d)$ such that $\gamma_n\ord\gamma_{n+1}$ for all $n\in\bbn$ and such that $\{\gamma_n\}_{n\in\bbn}\to \gamma$ uniformly, then $\gamma_n\den\gamma$ for all $n\in\bbn$.
\end{lemma}

In the next section, we will be forming alternating sequences $\alpha_1\ord \alpha_{1}'\equiv \alpha_2\ord \alpha_2'\equiv \cdots$ where the equivalences are given by small perturbations. The next lemma, which is proved using elementary real analysis, allows us to manage all of these perturbations simultaneously.

\begin{lemma}\label{infinitecompositionlemma}
Suppose $\{f_n\}_{n\in\bbn}$ is a sequence of increasing homeomorphisms $f_n:\ui\to\ui$ such that $\rho(f_n,id_{\ui})\leq ar^n$ for some $a>0$ and $|r|<1$. If $g_n=f_{n}^{-1}\circ f_{n-1}^{-1}\circ \cdots \circ f_{2}^{-1}\circ f_{1}^{-1}$ for all $n\in\bbn$, then $\{g_n\}_{n\in\bbn}$ converges uniformly to a continuous, non-decreasing surjection $g_{\infty}:\ui\to\ui$.
\end{lemma}

Next, we combine the previous two lemmas.

\begin{lemma}\label{alternatinglemma}
Suppose $\{\alpha_n\}_{n\in\bbn}$ and $\{\alpha_{n}'\}_{n\in\bbn}$ are sequences of Cantor paths in a metric space $(X,d)$, $\gamma:\ui\to X$ is a path, and $\{f_n\}_{n\in\bbn}$ is a sequence of increasing homeomorphisms $f_n:\ui\to\ui$ such that the following hold:
\begin{enumerate}
\item $\{\alpha_n\}_{n\in\bbn}\to \gamma$ uniformly,
\item $\alpha_n\ord\alpha_n'$ for all $n\in\bbn$,
\item $\alpha_{n+1}=\alpha_{n}'\circ f_n$ for all $n\in\bbn$,
\item there exists $a>0$ and $|r|<1$ such that $\rho(f_n,id_{\ui})\leq ar^n$ for all $n\in\bbn$.
\end{enumerate}
Then $\alpha_n\den\gamma$ for all $n\in\bbn$.
\end{lemma}

\begin{proof}
Let $g_0=id_{\ui}$ and $g_n=f_{n}^{-1}\circ f_{n-1}^{-1}\circ \cdots \circ f_{2}^{-1}\circ f_{1}^{-1}$ for all $n\in\bbn$. Note that $f_{n}\circ g_{n}=g_{n-1}$ for all $n\in\bbn$. By Lemma \ref{infinitecompositionlemma}, $\{g_n\}_{n\in\bbn}$ converges uniformly to a continuous, non-decreasing surjection $g_{\infty}:\ui\to\ui$.

Let $\gamma_1=\alpha_1$ and for $n\geq 2$, set $\gamma_n=\alpha_n\circ g_{n-1}$. For the moment, fix $n\in\bbn$. Since $\alpha_{n+1}=\alpha_{n}'\circ f_n$, we have \[\gamma_{n+1}=\alpha_{n+1}\circ g_{n}=\alpha_{n}'\circ f_n\circ g_n=\alpha_{n}'\circ g_{n-1}.\] By assumption, $\alpha_{n}\ord \alpha_{n}'$. Composing $\alpha_{n}$ and $ \alpha_{n}'$ with the homeomorphism $g_{n-1}$ gives $\gamma_{n}\ord\gamma_{n+1}$. Additionally, note that since $\{\alpha_n\}_{n\in\bbn}\to \gamma$ and $\{g_{n-1}\}_{n\in\bbn}\to g_{\infty}$ uniformly, we have $\{\gamma_n\}_{n\in\bbn}=\{\alpha_n\circ g_{n-1}\}_{n\in\bbn}\to \gamma\circ g_{\infty}$. It now follows from Lemma \ref{gammaextensionlemma} that $\gamma_n\den\gamma\circ g_{\infty}$ for all $n\in\bbn$. Thus $\alpha_n\circ g_{n-1}\den \gamma\circ g_{\infty}$ for all $n\in\bbn$. Since we have Fr\'echet equivalences  (see Definition \ref{deffrechet}) $\alpha_n\tripproxs \alpha_n\circ g_{n-1}$ and $\gamma\tripproxs \gamma\circ g_{\infty}$, it follows from Lemma \ref{denproperties} (4) that $\alpha_{n}\den \gamma$ for all $n\in\bbn$.
\end{proof}
\section{A Proof of the Main Theorem}\label{sec: proof of main theorem}

To begin, we fix Cantor-path parameterizations of planar line segments. Recall that $\tau:\ui\to\ui$ is the ternary Cantor map.
\begin{definition}\label{lxydef}
Given points $x,y$ in the closed unit disc $\bbd^2$, let $L_{x,y}:\ui\to\bbd^2$ be the path defined by $L_{x,y}(s)=\tau(s)(\frac{x+y}{2})+(1-\tau(s))x$.
\end{definition}

\begin{remark}\label{linearpathremark}
If $x=y$, then $L_{x,y}$ is constant. If $x\neq y$, then $L_{x,y}$ is a Cantor path that parameterizes the line segment from $x$ to $\frac{x+y}{2}$. Moreover, if $x_1,x_2,y_1,y_2\in \bbd^2$ with midpoints $m_i=\frac{x_i+y_i}{2}$, then since the paths $L_{x_1,y_1}$ and $\ov{L_{y_2,x_2}}$ are both parameterized using $\tau$, their sup-distance is the maximum distance between the endpoints, that is,
\begin{eqnarray*}
\rho(L_{x_1,y_1},\ov{L_{y_2,x_2}}) &\leq& \max\left\{d(x_1,m_2),d(m_1,y_2)\right\}\\
&\leq & \max\left\{\frac{d(x_2,y_2)}{2}+d(x_1,x_2),\frac{d(x_1,y_1)}{2}+d(y_1,y_2)\right\}
\end{eqnarray*}
(See Figure \ref{fig1}). In the case that $x=x_1=x_2$ and $y=y_1=y_2$, we have $\rho(L_{x,y},\ov{L_{y,x}})=\frac{1}{2}d(x,y)$.
\end{remark}

\begin{figure}[H]
\centering \includegraphics[height=2.3in]{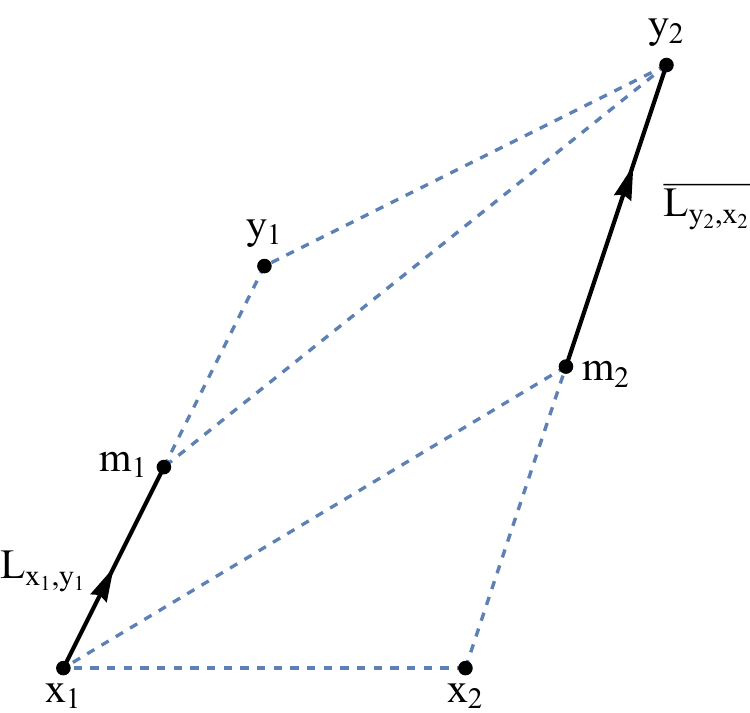}
\caption{\label{fig1}Since $L_{x_1,y_1}$ and $\ov{L_{y_2,x_2}}$ are both parameterized by $\tau$, their metric distance is the maximum length between starting and ending points.}
\end{figure}

We attribute the next lemma especially to the third author. The main idea is to fix staggered Cantor paths $\alpha,\beta:\ui\to \bbd^2$ satisfying $\alpha(i)=\beta(i)$, $i\in\{0,1\}$ and modify both of them by inserting CIP-loops on a $\bbz$-ordered sequence of elements of $\lc(\alpha)$ and $\lc(\beta)$ respectively so that $\alpha\ord\alpha '$ and $\beta\ord\beta '$. Specifically, we construct $\alpha '$ from $\alpha$ by inserting a bi-infinite sequence of CIP-loops of the form $L_{x,y} \ov{L_{x,y}}$ or $(L_{x,y}) c( \ov{L_{x,y}})$ (for constant path $c$) on certain elements of $\lc(\alpha)$. We will construct $\beta '$ from $\beta$ in an analogous way. However, the two constructions are not symmetric. Rather, they must be done simultaneously in an ``interlocking" fashion. In the end, the resulting paths $\alpha '$ and $\beta'$ will have $2/3$ the sup-distance of the original paths. That one can insert non-constant portions into both paths and somehow shrink the distance between them is somewhat non-intuitive and is only possible because the paths are \textit{staggered}.

Since the construction of $\alpha '$ and $\beta '$ will involve an intricate arrangement of overlapping intervals, we employ the following notation: if $I\subseteq \ui$ is an interval, then $\ell(I)$ and $r(I)$ will denote the left and right endpoints of $I$ respectively.

\begin{lemma}\label{fabelsconstruction}
For staggered Cantor paths $\alpha,\beta:\ui\to \bbd^2$ such that $\alpha(i)=\beta(i)$, $i\in\{0,1\}$, there exists Cantor paths $\alpha ',\beta ':\ui\to\bbd^2$ such that:
\begin{enumerate}
\item $\alpha\ord \alpha '$ and $\rho(\alpha,\alpha ')\leq \frac{1}{2}\rho(\alpha,\beta)$,
\item $\beta\ord \beta '$ and $\rho(\beta,\beta ')\leq \frac{1}{2}\rho(\alpha,\beta)$,
\item $\rho(\alpha ',\beta ')\leq  \frac{2}{3}\rho(\alpha,\beta)$.
\end{enumerate}
\end{lemma}

\begin{proof}
Let $\delta=\rho(\alpha,\beta)$. By assumption, $\mco(\alpha)$ and $\mco(\beta)$ are staggered Cantor complements. By Lemma \ref{choosezsetlemma}, we may select a set of connected components $\scru_1$ of $\mco(\alpha)$ and $\scru_2$ of $\mco(\beta)$ such that for $i\in\{0,1\}$, $\scru_i$ has the order type of $\bbz$, $\inf(\bigcup\scru_i)=0$, $\sup(\bigcup\scru_i)=1$, and such that $\bigcup\scru_1$ and $\bigcup\scru_2$ are staggered. Index the elements of $\scru_1$ as $\cdots <A_{-2}<A_{-1}<A_0<A_1<A_2<\cdots$ and the elements of $\scru_2$ as $\cdots <C_{-2}<C_{-1}<C_0<C_1<C_2<\cdots$ so that $C_{n}$ meets $A_{n-1}$ and $A_n$. Set $w_n=\alpha(\ov{A_n})$ and $y_n=\beta(\ov{C_n})$

For the moment, fix $n\in\bbz$. Since $\mco(\alpha)$ is a Cantor complement and $\alpha$ is uniformly continuous, we may find a sequence $B_{n,1}<B_{n,2}<\cdots< B_{n,k_n}$ in $\mco(\alpha)$ of length $k_n\geq 2$, where each set $B_{n,j}$ is contained in $[r(A_n), \ell(A_{n+1})]$ and such that if $I$ is a connected component of $[r(A_n), \ell(A_{n+1})]\backslash \bigcup_{j=1}^{k_n}B_{n,j}$, then $\diam(\alpha(\ov{I}))<\frac{\delta}{6}$. Similarly, since $\mco(\beta)$ is a Cantor complement, we may find a sequence $D_{n,1}<D_{n,2}<\cdots< D_{n,m_n}$ in $\mco(\beta)$ of length $m_n\geq 2$ where each $D_{n,j}$ is contained in $[r(C_{n-1}), \ell(C_{n})]$ and such that if $I$ is a connected component of $[r(C_{n-1}), \ell(C_{n})]\backslash \bigcup_{j=1}^{m_n}D_{n,j}$, then $\diam(\beta(\ov{I}))<\frac{\delta}{6}$. Note that $D_{n,j}\subseteq A_n$ and $B_{n,j}\subseteq C_n$ holds whenever these sets are defined (See Figure \ref{fig2}).
\begin{figure}[H]
\centering \includegraphics[height=.9in]{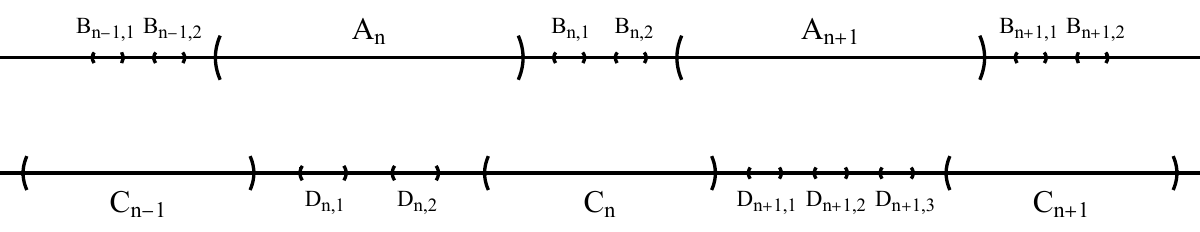}
\caption{\label{fig2} The upper segment illustrates the selected sets $A_n$ and $B_{n,j}$ on which $\alpha$ is constant. The lower segment illustrates the selected sets $C_n$ and $D_{n,j}$ on which $\beta$ is constant.}
\end{figure}
Set $x_{n,j}=\alpha(\ov{B_{n,j}})$ and $z_{n,j}=\beta(\ov{D_{n,j}})$ whenever these sets are defined. Additionally, let $p_n$ be the midpoint of $[\ell(A_n),r(C_{n-1})]$, $q_n$ be the midpoint of $[\ell(C_n),r(A_n)]$, $\eta_{n,j}$ be the midpoint of $B_{n,j}$, and $\theta_{n,j}$ be the midpoint of $D_{n,j}$.

To begin our definition of $\alpha'$, we first set $\alpha'$ to agree with $\alpha$ on $\{0,1\}$ and on $[r(A_n),\ell(A_{n+1})]\backslash \bigcup_{j=1}^{k_n}B_{n,j}$ for all $n\in\bbz$.
We complete the definition of $\alpha' $ piecewise by fixing $n\in\bbn$ and defining $\alpha'$ on $\ov{A_n}$ in three cases and $\ov{B_{n,j}}$ for $1\leq j\leq k_n$ in a fourth case. It may be helpful to note that
\[\ov{A_n}=[\ell(A_{n}),p_n]\cup [p_n,r(C_{n-1})]\cup [r(C_{n-1}),\ell(C_n)]\cup [\ell(C_n),q_n]\cup [q_n,r(A_n)]\]
(see Figure \ref{fig3}).

\begin{itemize}
    \item[$\bf (A1)$]\label{alpha'1} On $[\ell(A_{n}),\theta_{n,1}]$, we define $\alpha'$ to be the CIP-loop, which is the linear reparameterization of $L_{w_n,y_{n-1}}$ on $[\ell(A_{n}),p_n]$, the constant path at $\frac{w_n+y_{n-1}}{2}$ on $[p_n,\ell(D_{n,1})]$, and the linear reparameterization of $\ov{L_{w_n,y_{n-1}}}$ on $[\ell(D_{n,1}),\theta_{n,1}]$.
\vspace{.75em}
    \item[$\bf (A2)$]\label{alpha'2} On $[\theta_{n,j},\theta_{n,j+1}]$ (for each $1\leq j\leq m_{n}-1$), we define $\alpha '$ to be the CIP-loop, which is the linear reparameterization of $L_{w_n,z_{n,j}}$ on $[\theta_{n,j},r(D_{n,j})]$, the constant path at $\frac{w_n+z_{n,j}}{2}$ on $[r(D_{n,j}),\ell(D_{n,j+1})]$, and the linear reparameterization of $\ov{L_{w_n,z_{n,j}}}$ on $[\ell(D_{n,j+1}),\theta_{n,j+1}]$.
\vspace{.75em}
    \item[$\bf (A3)$]\label{alpha'3} On $[\theta_{n,m_n},r(A_n)]$, we define $\alpha '$ to be the CIP-loop, which is the linear reparameterization of $L_{w_n, z_{n,m_n}}$ on $[\theta_{n,m_n},r(D_{n,m_n})]$, the constant path at $\frac{w_n+ z_{n,m_n}}{2}$ on $[r(D_{n,m_n}),q_n]$, and the linear reparameterization of $\ov{L_{w_n, z_{n,m_n}}}$ on $[q_n,r(A_n)]$. This completes the definition of $\alpha '$ on $\ov{A_n}$.
\vspace{.75em}
    \item[$\bf (A4)$]\label{alpha'4} Lastly, on $\ov{B_{n,j}}$, we define $\alpha '$ to be the linear reparameterization of the CIP-loop $L_{x_{n,j},y_{n}}\ov{L_{x_{n,j},y_{n}}}$.
\end{itemize}

This completes the definition of $\alpha '$ (compare Figures \ref{fig3} and \ref{figconst}). Note that $\alpha'$ agrees with $\alpha$ everywhere except on a $\bbz$-ordered sequence of elements of $\lc(\alpha)$ on which CIP-loops replace constant loops. Hence, it is clear that $\alpha'$ is a well-defined function.
To begin our definition of $\beta'$, we first set $\beta'$ to agree with $\beta$ on $\{0,1\}$ and on $[r(C_{n-1}),\ell(C_n)]\backslash \bigcup_{j=1}^{m_n}D_{n,j}$ for each $n\in\bbz$. We complete the definition of $\beta' $ piecewise by fixing $n\in\bbn$ and defining $\beta'$ on $\ov{C_n}$ in five cases and $\ov{D_{n,j}}$ for $1\leq j\leq m_n$ in a sixth case.

\begin{itemize}
    \item[$\bf (B1)$]\label{beta'1} On $[\ell(C_n), q_n]$, we define $\beta'$ to be constant at $y_n$ (this happens to agree with the value of $\beta$).
\vspace{.75em}
    \item[$\bf (B2)$]\label{beta'2} On $[ q_n,\eta_{n,1}]$, we define $\beta'$ to be CIP-loops, which is the linear reparameterization of $L_{y_n,w_n}$ on $[ q_n,r(A_n)]$, the constant path at $\frac{y_n+w_n}{2}$ on $[r(A_n),\ell(B_{n,1})]$, and $\ov{L_{y_n,w_n}}$ on $[\ell(B_{n,1}),\eta_{n,1}]$.
\vspace{.75em}
    \item[$\bf (B3)$]\label{beta'3} On $[\eta_{n,j},\eta_{n,j+1}]$ (for each $ 1\leq j\leq k_n-1$), we define $\beta'$ to be CIP-loops, which is the linear reparameterization of $L_{y_n,x_{n,j}}$ on $[\eta_{n,j},r(B_{n,j})]$, the constant path at $\frac{y_n+x_{n,j}}{2}$ on $[r(B_{n,j}),\ell(B_{n,j+1})]$, and the linear reparameterization of $\ov{L_{y_n,x_{n,j}}}$ on $[\ell(B_{n,j+1}),\eta_{n,j+1}]$.
\vspace{.75em}
    \item[$\bf (B4)$]\label{beta'4} On $[\eta_{n,k_n},p_{n+1}]$, we define $\beta'$ to be the CIP-loop, which is the linear reparameterization of $L_{y_n,x_{n,k_n}}$ on $[\eta_{n,k_n},r(B_{n,k_n})]$, the constant path at $\frac{y_n+x_{n,k_n}}{2}$ on $[r(B_{n,k_n}),\ell(A_{n+1})]$, and the linear reparameterization of $\ov{L_{y_n,x_{n,k_n}}}$ on $[\ell(A_{n+1}),p_{n+1}]$.
\vspace{.75em}
    \item[$\bf (B5)$]\label{beta'5} On $[p_{n+1},r(C_n)]$, we define $\beta'$ to be constant at $y_n$ (this happens to agree with the value of $\beta$). This completes the definition of $\beta '$ on $\ov{C_n}$.
\vspace{.75em}
    \item[$\bf (B6)$]\label{beta'6} On $\ov{D_{n,j}}$ (for each $1\leq j\leq m_n$), we define $\beta'$ to be the linear reparameterization of the CIP-loop $L_{z_{n,j},w_{n}} \ov{L_{z_{n,j},w_{n}}}$.
\end{itemize}

This completes the definition of $\beta '$, which is a well-defined function (see Figures \ref{fig3} and \ref{figconst}).
\begin{figure}[H]
\centering \includegraphics[height=1.4in]{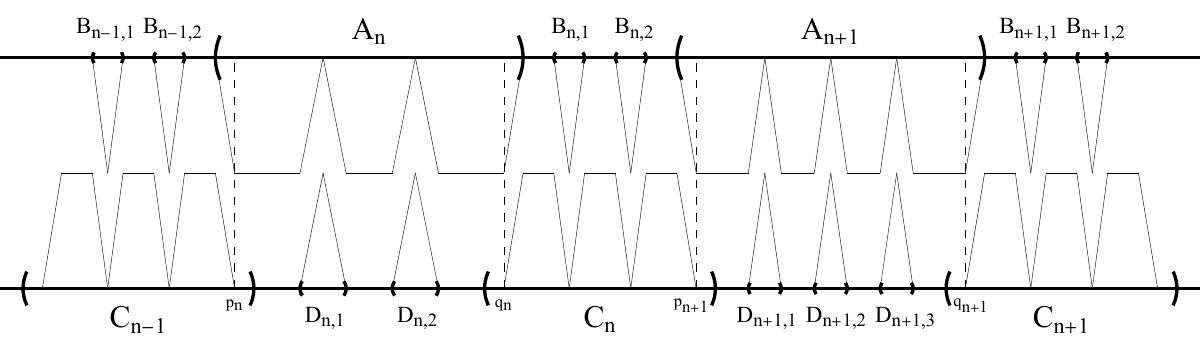}
\caption{\label{fig3} The interlocking pattern that determines the structure of $\alpha '$ and $\beta '$. The triangles represent inserted inverse-pair loops and the trapezoids represent an inverse pair loop but with a constant path included in the middle.  The three subintervals of $\ov{A_n}$ and the five subintervals of $\ov{C_n}$ partitioned by the trapezoids correspond to the piecewise definitions ${\bf(A1)}$-$\bf {(A3)}$ and ${\bf(B1)}$-$\bf {(B5)}$ respectively.}
\end{figure}
\begin{figure}[H]
\centering \includegraphics[height=1.6in]{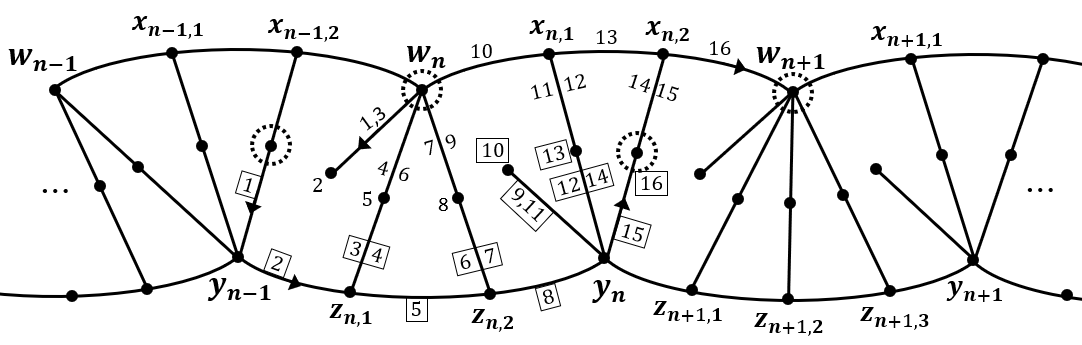}
\caption{\label{figconst} A full ``period" of the construction of $\alpha '$ and $\beta '$ starting with $\ell(A_n)$ and ending with $\ell(A_{n+1})$. Starting and ending points are circled and the initial and terminal steps are indicated with arrows. The upper and lower curves represent $\alpha$ and $\beta$ respectively where each subdivided segment has diameter less than $\frac{\delta}{6}$. The numbered paths trace out the trajectory of $\alpha'$ and the box-numbered paths trace out the corresponding trajectory of $\beta '$. When a number is positioned at a point, the path is constant at that point.}
\end{figure}

By construction, $\alpha '|_{(0,1)}$ is continuous. If $U$ is a convex neighborhood of $\alpha(0)=\beta(0)$ in $\bbd^2$, then we may find $N\in\bbz$ such that $\alpha(A_n\cup C_n)\cup \beta(A_n\cup C_n)\subseteq U$ for all $n\leq N$. Let $t=\sup(A_N)$. Since the CIP-loops added to $\alpha$ on $[0,t]$ are contained in line segments with endpoints in $\alpha([0,t])$ and $\beta([0,t])$ respectively, it follows that $\alpha '([0,t])\subseteq U$. Thus $\alpha '$ is continuous at $0$. A symmetric argument shows that $\alpha '$ is continuous at $1$. The construction of $\alpha '$ also ensures that distinct elements of $\lc(\alpha ')$ have disjoint closures. Hence, $\alpha '$ is a Cantor path. Additionally, since $\alpha '$ is constructed from  $\alpha$ only by replacing constant loops with CIP-loops on elements of $\lc(\alpha)$ (one on each $B_{n,j}$ and at least three on each $A_n$), we have $\alpha\ord \alpha '$. Moreover, if $\mu:\ov{I}\to \bbd^2$ is one of the added CIP-loops in the construction of $\alpha '$, then the image of $\mu$ is the line segment connecting $\alpha(s)$ and $\frac{\alpha(s)+\beta(s)}{2}$ for some $s\in \ov{I}$. Thus $\rho(\alpha,\alpha ')\leq \frac{1}{2}\rho(\alpha,\beta)$. Since these arguments apply just as well for $\beta '$, we also conclude that $\beta '$ is a Cantor path satisfying $\beta\ord\beta '$ and $\rho(\beta,\beta ')\leq \frac{1}{2}\rho(\alpha,\beta)$.

To complete the proof, we will show that $d(\alpha '(s),\beta '(s))\leq \frac{2\delta}{3}$. We begin by considering the case $s\in \ov{A_n}$.

\begin{enumerate}
    \item On $[\ell(A_n), p_n]$, $\alpha '$ parameterizes the line from $w_n$ to $\frac{w_n+y_{n-1}}{2}$  (See first part of $\bf (A1)$) and $\beta '$ is a corresponding parameterization of the line from $\frac{x_{n-1, k_{n-1}}+y_{n-1}}{2}$ to $y_{n-1}$  (See last part of $\bf (B4)$). Thus if $s\in [\ell(A_n), p_n]$, Remark \ref{linearpathremark} gives us the first inequality in the following sequence:\\
    \begin{eqnarray*}
     & & d(\alpha '(s),\beta '(s))  \\
    &\leq& \max\left\{\frac{d(x_{n-1, k_{n-1}},y_{n-1})}{2}+d(x_{n-1, k_{n-1}},w_n),\frac{d(w_n,y_{n-1})}{2}\right\}\\
    & =&  \max\left\{\frac{d(\alpha(r(B_{n-1, k_{n-1}})),\beta(r(B_{n-1, k_{n-1}})))}{2}+d(x_{n-1, k_{n-1}},w_n),\frac{d(\alpha(p_n),\beta(p_n))}{2}\right\}\\
    &<& \max\left\{\frac{\delta}{2}+\frac{\delta}{6},\frac{\delta}{2}\right\} \,\leq \, \frac{2\delta}{3}.
    \end{eqnarray*}

    \item On $[p_n,r(C_n)]$, $\alpha '$ is constant at $\frac{w_n+y_{n-1}}{2}$  (See middle part of $\bf (A1)$) and $\beta '$ is constant at $y_{n-1}$  (See $\bf (B5)$). Therefore, if $s\in [p_n,r(C_n)]$, we have \[d(\alpha '(s),\beta '(s))=\frac{d(w_n,y_{n-1})}{2}=\frac{d(\alpha(p_n),\beta(p_n))}{2}\leq\frac{\delta}{2}.\]

    \item On $[r(C_n),\ell(D_{n,1})]$, $\alpha '$ is constant at $\frac{w_n+y_{n-1}}{2}$  (See middle part of $\bf (A1)$) and $\beta '$ agrees with $\beta$. Thus if $s\in [r(C_n),\ell(D_{n,1})]$, then
    \begin{eqnarray*}
    d(\alpha '(s),\beta '(s)) &\leq & d(\alpha '(s),\beta(r(C_{n-1})))+d(\beta(r(C_{n-1})),\beta ' (s))\\
    &= & d\left(\frac{w_n+y_{n-1}}{2},y_{n-1}\right)+d(\beta(r(C_{n-1})),\beta  (s))\\
    &= & \frac{d(\alpha(p_n),\beta(p_n))}{2}+d(\beta(r(C_{n-1})),\beta  (s))\\
    &< & \frac{\delta}{2}+\frac{\delta}{6} \, \leq  \,\frac{2\delta}{3}
    \end{eqnarray*}

    \item On $[\ell(D_{n,1}),\theta_{n,1}]$, $\alpha '$ parameterizes the line segment from $\frac{w_n+y_{n-1}}{2}$ to $w_n$  (See last part of $\bf (A1)$) and $\beta '$ parameterizes the line segment from $z_{n,1}$ to $\frac{z_{n,1}+w_n}{2}$. In this case, if $s\in [\ell(D_{n,1}),\theta_{n,1}]$, then
    \begin{eqnarray*}
    d(\alpha '(s),\beta '(s)) &\leq&  \max\left\{ d\left(w_n,\frac{w_n+y_{n-1}}{2}\right),d\left(z_{n,1},\frac{z_{n,1}+w_n}{2}\right)\right\}\\
    &=& \max\left\{\frac{ d(w_n,y_{n-1})}{2},\frac{d(z_{n,1},w_n)}{2}\right\}\\
    &=& \max\left\{\frac{ d(\alpha(p_n),\beta(p_n))}{2},\frac{d(\alpha(\theta_{n,1}),\beta(\theta_{n,1}))}{2}\right\}\\
    & \leq & \frac{\delta}{2} \,<\, \frac{2\delta}{3}
    \end{eqnarray*}

    \item On $[\theta_{n,1},r(D_{n,1})]$, $\alpha '$ parameterizes the line segment from $w_n$ to $\frac{z_{n,1}+w_n}{2}$  (See first part of $\bf (A2)$) and $\beta$ parameterizes the line segment from $\frac{z_{n,1}+w_n}{2}$ to $z_{n,1}$  (See $\bf (B6)$). Since these paths move along the same line segment, it follows that if $s\in [\theta_{n,1},r(D_{n,1})]$, then \[d(\alpha '(s),\beta '(s))=\frac{d(w_n,z_{n,1})}{2}\leq \frac{d(\alpha(\theta_{n,1}),\beta(\theta_{n,1}))}{2}\leq \frac{\delta}{2}\]
\end{enumerate}

As we proceed through the remainder of the intervals on which $\alpha '$ and $\beta '$ are defined piecewise, every remaining case (all of which are illustrated in Figures \ref{fig3} and \ref{figconst}) may be verified using an argument nearly identical to one of the above five cases. Hence, we omit the remainder of the details. We conclude that $\rho(\alpha',\beta')\leq \frac{2\delta}{3}$.
\end{proof}

The construction given in the proof of Lemma \ref{fabelsconstruction} results in two Cantor paths $\alpha '$ and $\beta '$, which are not staggered. Hence, to iterate this construction, we must perturb one of these two paths so that they become staggered.

\begin{lemma}\label{recursionlemma}
For given staggered Cantor paths $\alpha,\beta:\ui\to\bbd^2$ with $\alpha(i)=\beta(i)$ for $i\in\{0,1\}$ and $\delta=\rho(\alpha,\beta)$, there exists sequences of Cantor paths $\{\alpha_n\}_{n\geq 0}$, $\{\alpha_{n}'\}_{n\geq 0}$, $\{\beta_n\}_{n\geq 0}$, such that $\alpha_0=\alpha$, $\beta_0=\beta$ and such that for all $n\geq 0$,
\begin{enumerate}
\item $\alpha_{n}\ord \alpha_{n}'$ and $\beta_{n}\ord \beta_{n+1}$,
\item $\alpha_n$ and $\beta_n$ are staggered,
\item $\alpha_{n+1}=\alpha_{n}'\circ f_n$ for some increasing homeomorphism $f_n:\ui\to\ui$ with $\rho(f_n,id_{\ui})<\frac{1}{2^n}$,
\item $\max\{\rho(\alpha_{n}',\alpha_{n+1}),\rho(\alpha_n,\alpha_{n+1}),\rho(\beta_n,\beta_{n+1}),\rho(\alpha_n,\beta_{n})\}\leq \delta\left(\frac{3}{4}\right)^n$.
\end{enumerate}
\end{lemma}

\begin{proof}
Let $\alpha_0=\alpha$, $\beta_0=\beta$, and $\delta=\rho(\alpha_0,\beta_0)$. Suppose $\alpha_n$ and $\beta_n$ are given staggered Cantor paths that satisfy $\rho(\alpha_n,\beta_n)\leq \delta\left(\frac{3}{4}\right)^{n}$. Applying Lemma \ref{fabelsconstruction}, find Cantor paths $\alpha_{n}',\beta_{n}':\ui\to \bbd^2$ such that $\alpha_n\ord \alpha_{n}'$, $\beta_{n}\ord \beta_{n}'$, $\max\{\rho(\alpha_n,\alpha_{n}'),\rho(\beta_n,\beta_{n}')\}\leq \frac{1}{2}\rho(\alpha_n,\beta_n)$, and $\rho(\alpha_{n}',\beta_{n}')\leq \frac{2\delta}{3}$. By Lemma \ref{staggeredlemma}, there exists an increasing homeomorphism $f_n:\ui\to \ui$ such that $\rho(f_n,id_{\ui})<\frac{1}{2^n}$, $\rho(\alpha_{n}'\circ f_n,\alpha_{n}')<\frac{1}{12}\rho(\alpha_n,\beta_n)$, and such that $\alpha_{n}'\circ f_n$ and $\beta_{n}'$ are staggered. Set $\alpha_{n+1}=\alpha_{n}'\circ f_n$ and $\beta_{n+1}=\beta_{n}'$. Then $\alpha_{n+1}$ and $\beta_{n+1}$ are staggered Cantor paths and satisfy the following inequalities:
\[\rho(\alpha_{n}',\alpha_{n+1})<\frac{1}{12}\rho(\alpha_n,\beta_n)\leq \frac{\delta}{12}\left(\frac{3}{4}\right)^{n}\leq \delta\left(\frac{3}{4}\right)^{n+1}\]
and
\begin{eqnarray*}
\rho(\alpha_n,\alpha_{n+1}) &\leq& \rho(\alpha_n,\alpha_{n}')+\rho(\alpha_{n}',\alpha_{n+1})\\
&<&
\frac{1}{2}\rho(\alpha_n,\beta_n)+\frac{1}{12}\rho(\alpha_n,\beta_n)\\
&<& \frac{3}{4}\rho(\alpha_n,\beta_n) \,\leq \, \delta\left(\frac{3}{4}\right)^{n+1}
\end{eqnarray*}
and
\[\rho(\beta_n,\beta_{n+1}) \leq \frac{1}{2}\rho(\alpha_n,\beta_n)\leq \frac{3}{4}\rho(\alpha_n,\beta_n)\leq  \delta\left(\frac{3}{4}\right)^{n+1} \]
and
\begin{eqnarray*}
\rho(\alpha_{n+1},\beta_{n+1}) &\leq& \rho(\alpha_{n+1},\alpha_{n}')+\rho(\alpha_{n}',\beta_{n}')\\
&\leq& \frac{1}{12}\rho(\alpha_n,\beta_n)+\frac{2}{3}\rho(\alpha_n,\beta_n)\\
&= & \frac{3}{4}\rho(\alpha_n,\beta_n) \,\leq \, \delta\left(\frac{3}{4}\right)^{n+1}.
\end{eqnarray*}
This completes the inductive construction of the desired sequences.
\end{proof}

In the next two statements, we assume $\alpha,\beta$ are fixed staggered Cantor paths as given in the hypothesis of Lemma \ref{recursionlemma}.

\begin{proposition}\label{gammaprop}
The sequences $\{\alpha_n\}_{n\geq 0}$ and $\{\beta_n\}_{n\geq 0}$ constructed in the proof of Lemma \ref{recursionlemma} both converge uniformly to a single path $\gamma:\ui\to \bbd^2$.
\end{proposition}

\begin{proof}
Recall that $\delta=\rho(\alpha_0,\beta_0)$ is fixed. Since $\rho(\alpha_n,\alpha_{n+1})\leq \delta\left(\frac{3}{4}\right)^n$ for all $n\geq 0$, $\{\alpha_n\}_{n\geq 0}$ is Cauchy in the $\sup$-metric and, therefore, converges uniformly to some path $\gamma:\ui\to \bbd^2$. Since the sequence $\{\beta_n\}_{n\geq 0}$ satisfies the same inequality, $\{\beta_n\}_{n\geq 0}$ also converges uniformly to some path. Additionally, since $\rho(\alpha_{n},\beta_{n})\leq \delta(\frac{3}{4})^{n}$ for all $n\geq 0$, $\{\alpha_n\}_{n\geq 0}$ and $\{\beta_n\}_{n\geq 0}$ must both converge uniformly to $\gamma$.
\end{proof}

\begin{lemma}\label{denlemma}
If $\gamma:\ui\to \bbd^2$ is the uniform limit of the sequences $\{\alpha_n\}_{n\geq 0}$ and $\{\beta_n\}_{n\geq 0}$ as given in the conclusion of Proposition \ref{gammaprop}, then $\alpha\den\gamma$ and $\beta\den\gamma$.
\end{lemma}

\begin{proof}
For given paths $\alpha$ and $\beta$, Lemma \ref{recursionlemma} gives sequences $\{\alpha_n\}_{n\geq 0}$, $\{\alpha_n'\}_{n\geq 0}$, $\{f_n\}_{n\in\bbn}$, and $\{\beta_n\}_{n\geq 0}$ satisfying a variety of relations and inequalities. Proposition \ref{gammaprop} ensures $\{\alpha_n\}_{n\geq 0}$ and $\{\beta_n\}_{n\geq 0}$ converge uniformly to a path $\gamma$. The sequences $\{\alpha_n\}_{n\geq 0}$, $\{\alpha_n'\}_{n\geq 0}$, $\{f_n\}_{n\in\bbn}$ and the limit path $\gamma$ satisfy the hypotheses of Lemma \ref{alternatinglemma}. It follows that $\alpha_n\den\gamma$ for all $n\geq 0$. In particular, $\alpha=\alpha_0\den\gamma$. Similarly, we may apply Lemma \ref{alternatinglemma} to the sequence $\{\beta_n\}_{n\geq 0}$ in the case where $\beta_{n}=\beta_n'$ and $f_n=id$ for all $n\geq 0$ (or we could apply Lemma \ref{gammaextensionlemma}). Thus, $\beta=\beta_0\den \gamma$.
\end{proof}

\subsection{Proof of Theorem \ref{theoremmain}}

\begin{proof}[Proof of Theorem \ref{theoremmain}]
First, we prove  the special case $X=\bbd^2$. Respectively, let $a$ and $b$ be the injective paths in $\bbd^2$ from $(1,0)$ to $(-1,0)$ that parameterize the upper and lower semicircles of $S^1$ respectively. Let $\tau:\ui\to\ui$ be the ternary Cantor map and note that $a\circ \tau$ and $b\circ \tau$ are Cantor paths. Set $b_0=b\circ \tau$. By Lemma \ref{staggeredlemma} (taking $\epsilon=1$), there exists an increasing homeomorphism $f:\ui\to \ui$ such that $a_0=a\circ \tau\circ f$ and $b_0$ are staggered. Applying Lemma \ref{denlemma} to the pair of staggered Cantor paths $a_0$ and $b_0$, we obtain the existence of a path $c:\ui\to \bbd^2$ such that $a_0\den c$ and $b_0\den c$. Since $a\tripproxs a_0$ and $b \tripproxs b_0$, Lemma \ref{denproperties} (4) gives that $a\den c$ and $b\den c$. Since $a$ and $b$ are injective paths, it follows that $c\geq_{\bbr}a$ and $c\geq_{\bbr}b$

In the general case, suppose $\alpha ,\beta :\ui\to X$ are path homotopic. Find a map $f:\bbd^2\to X$ such that $f\circ a=\alpha$ and $f\circ b=\beta$. By Lemma \ref{denproperties} (1), the path $\gamma =f\circ c$ in $X$ satisfies $\gamma\geq_{\bbr}\alpha$ and $\gamma\geq_{\bbr} \beta$.
\end{proof}

\begin{corollary}\label{corcoincideswithhomotopy}
The equivalence relation on the set of paths in a given topological space generated by $\geq_{\bbr}$ (and $\den$) coincides with path-homotopy.
\end{corollary}
 If $X$ is one-dimensional, then a loop is  null-homotopic if and only if it factors through a loop in a dendrite, see \cite[Theorem 3.7]{CConedim} for the nontrivial implication.  Thus, for one-dimensional spaces, the  relation $\den$ (see Definition \ref{defden}) is equivalent to the homotopy rel.\ endpoints relation and hence is transitive. Since $\bbd^2$ contains simple closed curves (parameterizations of which cannot factor through an $\bbr$-tree),  Theorem \ref{theoremmain} implies the following.

\begin{corollary}\label{corollarynottransitive}
If $\bbd^2$ embeds in $X$, then the $\bbr$-tree homotopy relation on the set of paths in $X$ is not transitive.
\end{corollary}

\section{A Solution to Dydak's Problem}\label{sec: Dydak's problem}

We conclude with a proof of Lemma \ref{lemdydaksproblem} and Theorem \ref{orbit}. First, we note the following lemma, which is proved using standard techniques from covering space theory.

\begin{lemma}\label{dendriteliftlemma}
Let $p:(E,e_0)\to (X,x_0)$ be a based map with unique lifting of paths rel.\ starting point and suppose $T$ is an $\bbr$-tree. If $f:(T,t_0)\to (X,x_0)$ is a based map, then there exists a unique based map $\wt{f}:(T,t_0)\to (E,e_0)$ such that $p\circ \wt{f}=f$.
\end{lemma}

\subsection{Proof of Lemma \ref{lemdydaksproblem}}

\begin{proof}[Proof of Lemma \ref{lemdydaksproblem}]
Let $p:E\to X$ be a map where $X$ is first countable, locally path connected and simply connected and such that every path in $X$ has a unique lift in $E$ rel.\ starting point. Since $X$ is path connected, $p$ is surjective. Suppose that $p(e_1)=p(e_2)=x$ for $e_1,e_2\in E$. Let $\wt{\beta}:\ui\to E$ be a path from $e_1$ to $e_2$. Then $\beta=p\circ\wt{\beta}$ is a loop based at $x$. Let $\alpha:\ui\to X$ be the constant path at $x$. Since $X$ is simply connected, $\alpha\simeq \beta$ and thus, by Theorem \ref{theoremmain}, there exists a path $\gamma:\ui\to X$ such that $\gamma\geq_{\bbr}\alpha$ and $\gamma\geq_{\bbr} \beta$. Let $\wt{\alpha},\wt{\gamma}:(\ui,0)\to (E,e_1)$ be the lifts of $\alpha$ and $\gamma$ respectively and note that $\wt{\alpha}$ is constant at $e_1$. Since $\alpha\ov{\gamma}$ factors through an $\bbr$-tree, it follows from Lemma \ref{dendriteliftlemma} that $\wt{\alpha}(1)=\wt{\gamma}(1)$ in $E$. Similarly, we have $\wt{\beta}(1)=\wt{\gamma}(1)$. Thus $e_1=\wt{\gamma}(1)=e_2$, proving that $p$ is injective. Since $X$ is first countable, it suffices to show $p^{-1}:X\to E$ preserves convergent sequences. If $\{x_n\}\to x$ is a convergent sequence in $X$, the hypotheses on $X$ allow us to find a path $\alpha:\ui\to X$ with $\alpha(1/n)=x_n$ for all $n\in\bbn$ and $\alpha(0)=x$. There is one lift $\wt{\alpha}:\ui\to E$ for which $p\circ\wt{\alpha}=\alpha$ and it satisfies $\wt{\alpha}(1/n)=p^{-1}(x_n)$ and $\wt{\alpha}(0)=p^{-1}(x)$. Since $\wt{\alpha}$ is continuous, $\{p^{-1}(x_n)\}\to p^{-1}(x)$ in $E$. Thus $p^{-1}$ is continuous.
\end{proof}

\begin{proof}[Proof of Theorem \ref{orbit}]
  Suppose that a group $G$ acts on an $\mathbb R$-tree $T$ such that the quotient map $p:T \to T/G$ is a UPL map.  Suppose that $T/G$ contains a 2-dimensional Euclidean disc $D$.  Let $E$ be a path-component of $p^{-1}(D)$.  Then $p|_{E}:E \to D$ is a UPL map over a first countable, locally path-connected, and simply connected space $D$.  Hence $p|_{E}$ is a homeomorphism which is a contradiction, since $E$ is a subspace of an $\mathbb R$-tree.
\end{proof}

\subsection{Proof of Corollary \ref{corcoveringmap}}

\begin{proof}[Proof of Corollary \ref{corcoveringmap}]
    Notice that Lemma \ref{lemdydaksproblem} implies that Dydak's Unique Lifting Problem has a positive answer.  It is an exercise to see that the evaluation map from $P(X,x)$ to $X$ is an open surjection, since $X$ is locally path-connected, see \cite[Section 2.1]{BM}.  Then Corollary \ref{corcoveringmap} follows immediately from Corollary 4.10 of \cite{BM}.
\end{proof}

\section{Appendix: A Proof of Lemma \ref{gammaextensionlemma}}\label{appendix}

To prove Lemma \ref{gammaextensionlemma}, we must first detail the structure of a single $\scru$-extension. For the moment, suppose that $\beta:\ui\to X$ is a given $\scru$-extension of $\alpha:\ui\to X$ (where both are Cantor paths). Additionally, we fix
\begin{enumerate}
\item an $\lc(\alpha)$-collapsing map $k_{\alpha}$ and light path $\alpha^{\light}$ such that $\alpha^{\light}\circ k_{\alpha}=\alpha$,
\item an $\lc(\beta)$-collapsing map $k_{\beta}$ and light path $\beta^{\light}$ such that $\beta^{\light}\circ k_{\beta}=\beta$.
\end{enumerate}
It follows from Remark \ref{extremark} that there is a unique map $k_{\beta,\alpha}:\ui\to \ui$ such that $k_{\alpha}=k_{\beta,\alpha}\circ k_{\beta}$. Let $\scrv=\{k_{\beta}(J)\mid J\in\scru\}$ be the collection of open intervals which are the images of the elements of $\scru$. If $J\in \scru$ and $(c,d)=k_{\beta}(J)\in \scrv$, then $\beta|_{\ov{J}}$ is a CIP-loop and $(\beta^{\light})|_{[c,d]}$ is a LIP-loop. Thus we have equivalence $(\beta^{\light})|_{[c,d]}\equiv b_{J} \ov{b_{J}}$ for a light path $b_J$. In particular, there exists $m\in (c,d)$ and a ``tent map" $\tau_J:[c,d]\to\ui$ which (1) is an increasing homeomorphism $[c,m]\to [0,1]$ on $[c,m]$, (2) a decreasing homeomorphism $[m,d]\to[0,1]$ on $[m,d]$, and (3) satisfies the equality $(\beta^{\light})|_{[c,d]}=b_J\circ \tau_J$.

Let $\sim$ be the smallest equivalence relation on $\ui$ satisfying the following: identify $s\sim t$ if there exists $J\in\scru$ such that $s,t\in [c,d]=k_{\beta}(\ov{J})$ and $\tau_J(s)=\tau_J(t)$. Let $D(\beta,\alpha)=\ui/\mathord{\sim}$ and $q:\ui\to D(\beta,\alpha)$ denote the quotient ``folding" map.

Note that $D(\beta,\alpha)$ is constructed by folding each interval $[c,d]=k_{\beta}(\ov{J})$, $J\in\scru$ ``in half" according to the tent map $\tau_J$. Hence, $D(\beta,\alpha)$ is a dendrite consisting of the ``base arc" $B=q(\ui\backslash \bigcup\scrv)$ and possibly infinitely many attached arcs. In particular, the arc $A_J=q([c,d])$ meets $B$ at the point $q(\{c,d\})$ and has free-endpoint $q(m)$ (where $m$ is defined as above).

The definition of $\sim$ ensures that $\beta^{\light}$ is constant on the fibers of $q$ and, therefore, there is a unique map $F:D(\beta,\alpha)\to X$ such that $F\circ q=\beta^{\light}$. Recall that each fiber of $k_{\beta}$ is contained in (and possibly equal to) a fiber of $k_{\alpha}$. Moreover if $q(s)=q(t)$ for $s\neq t$, then $k_{\beta}^{-1}(s)$ and $k_{\beta}^{-1}(t)$ lie in the same fiber of $k_{\alpha}$ (the closure of some element of $\scru$). Therefore, $k_{\alpha}$ is constant on the fibers of $q\circ k_{\beta}$ and there is a unique map $r:D(\beta,\alpha)\to \ui$ such that $k_{\alpha}=r\circ q\circ k_{\beta}$. In particular, $r$ maps the base-arc $B$ homeomorphically onto $\ui$ and if $(c,d)=k_{\beta}(J)$ for $J\in\scru$, then $r$ maps the arc $A_J$ to the point $k_{\beta,\alpha}([c,d])=k_{\alpha}(\ov{J})$ in $\ui$. It follows that $r$ is a monotone map. Finally, since $r\circ q\circ k_{\beta}=k_{\alpha}=k_{\beta,\alpha}\circ k_{\beta}$ where $k_{\beta}$ is surjective, we have $r\circ q=k_{\beta, \alpha}$. Overall, the following diagram on the left commutes. We address the diagram on the right in the next proposition.
\[\xymatrix{
\ui \ar@/_4pc/[ddd]_-{\beta} \ar[d]^-{k_{\beta}} \ar[r]^-{id} & \ui \ar[d]_-{k_{\alpha}} \ar@/^3pc/[ddd]^(.7){\alpha} &&&  C \ar@/_4pc/[ddd]_(.3){\beta|_{C}} \ar[d]^-{k_{\beta}|_{C}} \ar[r]^-{id} & C \ar[d]_-{k_{\alpha}|_{C}} \ar@/^3pc/[ddd]^-{\alpha|_{C}}   \\
\ui \ar@/_2pc/[dd]_(.3){\beta^{\light}} \ar[d]^-{q} \ar[r]^-{k_{\beta,\alpha}} & \ui \ar[d]_-{id} &&&  \ui \ar@/_2pc/[dd]_(.3){\beta^{\light}} \ar[d]^-{q}  & \ui \ar[dl]^-{\sigma} \ar[d]^-{id}   \\
D(\beta,\alpha) \ar[d]^-{F} \ar[r]^-{r} & \ui \ar[d]_-{\alpha^{\light}} &&& D(\beta,\alpha) \ar[d]^-{F}  & \ui \ar[d]_-{\alpha^{\light}}\\
X & X &&&  X \ar@{=}[r] & X
}\]

\begin{proposition}\label{retractionprop}
The map $r:D(\beta,\alpha)\to \ui$ is a retraction. Moreover, the unique section $\sigma:\ui\to D(\beta,\alpha)$, which satisfies $r\circ\sigma=id$, parameterizes the base-arc $B$ and satisfies $F\circ \sigma=\alpha^{\light}$.
\end{proposition}

\begin{proof}
Let $C=\ui\backslash \bigcup\scru$ and note that the restrictions $q\circ k_{\beta}|_{C}:C\to B$ and $(k_{\alpha})|_{C}:C\to \ui$ are quotient maps that make the same identifications, namely they collapse intervals $\ov{J}$ for $J\in\lc(\alpha)\backslash \scru$ and identify endpoints of intervals $J\in\scru$. Thus, there exists a  unique homeomorphism $\sigma:\ui\to B$ such that $\sigma\circ (k_{\alpha})|_{C}=q\circ k_{\beta}|_{C}$ (see the right diagram above). Since $r\circ \sigma\circ (k_{\alpha})|_{C}=r\circ q\circ k_{\beta}|_{C}=(k_{\alpha})|_{C}$ where $(k_{\alpha})|_{C}$ is surjective, we have $r\circ \sigma=id_{\ui}$. Overall, we have that $r$ is a retraction whose section $\sigma$ parameterizes $B$.

One can use the above diagram to confirm that $F\circ \sigma\circ (k_{\alpha})|_{C}=\beta|_{C}$. By definition of $\beta$ being a $\scru$-extension of $\alpha$, we have $\alpha|_{C}=\beta|_{C}$. Thus $F\circ \sigma\circ (k_{\alpha})|_{C}=\alpha|_{C}=\alpha^{\light}\circ (k_{\alpha})|_{C}$. Since $(k_{\alpha})|_{C}$ is surjective, we have the desired equality $F\circ \sigma=\alpha^{\light}$.
\end{proof}

\begin{corollary}
If $\alpha\ord\beta$, then $\alpha\simeq_{\bbr}\beta$ and $\alpha^{\lambda}\leq_{\bbr}\beta^{\lambda}$.
\end{corollary}

\begin{proof}
Note that $\sigma\circ k_{\alpha}$ and $q\circ k_{\beta}$ are paths in the compact $\bbr$-tree $D(\beta,\alpha)$ with the same endpoints and which satisfy $F\circ (\sigma\circ k_{\alpha})=\alpha$ and $F\circ (q\circ k_{\beta})=\beta$. Thus $\alpha\den\beta$. Moreover, since $F\circ q=\beta^{\lambda}$ and $F\circ \sigma=\alpha^{\lambda}$ where $\sigma$ is injective, we have $\alpha^{\lambda}\leq_{\bbr}\beta^{\lambda}$.
\end{proof}

Again, we suppose that $\beta$ is a $\scru$-extension of $\alpha$ (where both are Cantor paths) and we reuse the above notation. However, now we suppose also that there exists a dendrite $D$, a map $H:D\to X$, and a surjective path $Q:\ui\to D$ such that $H\circ Q= \alpha^{\light}$ (surjectivity of $Q$ is not required for the following construction but appears naturally in our recursive application so we assume it). In particular, we use this data and the above construction of $D(\beta,\alpha)$ to uniquely determine a factorization of $\beta^{\lambda}$.

Let $D'$ be the pushout of $Q:\ui\to D$ and the section $\sigma:\ui\to D(\beta,\alpha)$. We have the following pushout square. Note that $i$ is injective since $\sigma$ is injective. Also, $j$ is surjective since $Q$ is surjective. All of the domains being compact ensures that $i$ is an embedding and $j$ is a quotient map. Observe that since $i(D)$ is a dendrite and $D'\backslash i(D)\cong D(\beta,\alpha)\backslash \sigma(\ui)$ is a disjoint union of half-open arcs, $D'$ is a dendrite.
\[\xymatrix{
\ui \ar[r]^{Q} \ar@{^{(}->}[d]_{\sigma} & D \ar@{^{(}->}[d]^-{i} \\
D(\beta,\alpha) \ar[r]_-{j}  & D'
}\]
Since $H\circ Q=\alpha^{\light}=F\circ \sigma$, there exists a unique map $H':D'\to X$ making the left diagram below commute. Set $Q'=j\circ q:\ui\to D'$. Then $H'\circ Q'=H' \circ j\circ q=F\circ q=\beta^{\light}$. Additionally, since $r\circ \sigma=id_{\ui}$ the right diagram below shows that there exists a map $R:D'\to D$ such that $R\circ i=id_{D}$ and $R\circ j=Q\circ r$. Thus $R$ is a retraction with section $i$. Additionally, note that $R$ is monotone, since $r$ is monotone.
\[\xymatrix{
& \ui \ar@/^5pc/[ddrr]^-{\alpha^{\lambda}} \ar[r]_{Q}   \ar@{^{(}->}[d]_{\sigma} & D \ar@/^1pc/[ddr]^-{H} \ar@{^{(}->}[d]^-{i} && \ui \ar[r]^{Q}   \ar@{^{(}->}[d]_{\sigma} & D \ar@/^1pc/[ddr]^-{id_{D}} \ar@{^{(}->}[d]^-{i}
\\
\ui \ar[r]^-{q} \ar@/_2pc/[drrr]_-{\beta^{\lambda}} & D(\beta,\alpha) \ar@/_1pc/[drr]_-{F} \ar[r]_-{j}  & D' \ar@{-->}[dr]^-{H'} && D(\beta,\alpha) \ar@/_1pc/[drr]_-{Q\hspace{.125em}\circ \hspace{.125em} r } \ar[r]_-{j}  & D' \ar@{-->}[dr]^-{R}
 \\
&&& X & && D
}\]


We also have $R\circ Q'=R\circ j\circ q=Q\circ r\circ q=Q\circ k_{\beta,\alpha}$ and $H'\circ i\circ Q=\alpha^{\light}$. Finally, note that the paths $i\circ Q\circ k_{\alpha}:\ui\to D'$ and $Q'\circ k_{\beta}:\ui\to D'$ start and end at the same points and that $H'\circ i\circ Q\circ k_{\alpha}=\alpha$ and $H'\circ Q'\circ k_{\beta}=\beta$. This gives a factorization of the loop $\alpha\ov{\beta}:\ui\to X$ through the dendrite $D'$. Overall, we conclude the following, which employs the notation in the construction of both $D(\beta,\alpha)$ and $D'$.

\begin{lemma}\label{iterationlemma}
Let Cantor-path $\beta$ be a $\scru$-extension of another Cantor-path $\alpha$ and let $k_{\alpha}$ and $k_{\beta}$ be collapsing maps for these paths respectively. Suppose there exists a dendrite $D$, a map $H:D\to X$, and a surjective path $Q:\ui\to D$ such that $H\circ Q= \alpha^{\light}$. Then there exists a dendrite $D'$ (constructed from the pushout square $i\circ Q=j\circ \sigma$), a map $Q':\ui\to D'$ defined as $Q=j\circ q$, a monotone retraction $R:D'\to D$ with section $i$, and a map $H':D'\to X$ such that the following diagram commutes
\[\xymatrix{
\ui \ar@/_4pc/[ddd]_-{\beta} \ar[d]_-{k_{\beta}} \ar[r]^-{id} & \ui \ar[d]^-{k_{\alpha}} \ar@/^4pc/[ddd]^-{\alpha} \\
\ui \ar@/_2pc/[dd]_-{\beta^{\light}} \ar[d]^-{Q'} \ar[r]^-{k_{\beta,\alpha}} & \ui \ar[d]_-{Q}  \ar@/^2pc/[dd]^-{\alpha^{\light}}\\
D' \ar[d]^-{H'} \ar[r]^-{R} & D \ar[d]_-{H}\\
X & X
}\]
Moreover, $(i\circ Q\circ k_{\alpha}) (\ov{Q'\circ k_{\beta}})$ is a well-defined loop in $D'$ satisfying \[\alpha\ov{\beta}=H'\circ ((i\circ Q\circ k_{\alpha}) (\ov{Q'\circ k_{\beta}})).\]
\end{lemma}

In the proof of Lemma \ref{gammaextensionlemma}, we iterate the above construction. 


\begin{proof}[Proof of Lemma \ref{gammaextensionlemma}]
Since $\gamma_n\ord\gamma_{n+1}$, $\gamma_{n+1}$ is a $\scru_n$-extension of $\gamma_{n}$ for some $\scru_n\subseteq \lc(\gamma_n)$. Before we inductively apply the construction from Lemma \ref{iterationlemma}, we fix collapsing functions $k_{\gamma_n}$ for the paths $\gamma_n$. We then have uniquely determined maps $k_{\gamma_{n+1},\gamma_n}:\ui\to\ui$ such that $k_{\gamma_{n+1},\gamma_n}\circ k_{\gamma_{n+1}}=k_{\gamma_n}$ and light  paths $\gamma_{n}^{\lambda}$. To simplify notation in the inverse systems to come, we write $k_n$ for $k_{\gamma_n}$ and $K_{n+1,n}$ for $k_{\gamma_{n+1},\gamma_n}$. Recall that since $\gamma_{n}\ord\gamma_{n+1}$ for all $n\in\bbn$, we have a dendrite $D(\gamma_{n+1},\gamma_{n})$, which comes equipped with a corresponding folding map $q_{n+1}:\ui\to D(\gamma_{n+1},\gamma_{n})$, retraction $r_{n+1,n}:D(\gamma_{n+1},\gamma_{n})\to\ui$, and embedding $\sigma_{n,n+1}:\ui\to D(\gamma_{n+1},\gamma_{n})$ such that $r_{n+1,n}\circ q_{n+1}=k_{n+1,n}$ and $r_{n+1,n}\circ \sigma_{n,n+1}=id_{\ui}$.

To begin the recursion, set $D_1=\ui$, $H_1=\gamma_{1}^{\light}$, and $Q_1=id_{\ui}$ so that $H_1\circ k_{1}=\gamma_1$. Suppose that we have given dendrite $D_n$, map $H_n:D_n\to X$, and quotient map $Q_n:\ui\to D_n$ such that $H_n\circ Q_n=\gamma_{n}^{\light}$. We apply the construction used in the proof of Lemma \ref{iterationlemma} to the case where $\gamma_{n+1}$ is a $\scru_n$-extension of $\gamma_n$. We obtain a dendrite $D_{n+1}$ constructed as the pushout of embedding $\sigma_{n,n+1}$ and $Q_{n}$. This pushout construction yields a quotient map $j_{n+1}:D(\gamma_{n+1},\gamma_n)\to D_{n+1}$, and an embedding $i_{n,n+1}:D_n\to D_{n+1}$  such that  $i_{n,n+1}\circ Q_n=j_{n+1}\circ \sigma_{n,n+1}$. We also obtain a map $H_{n+1}:D_{n+1}\to X$, a quotient map $Q_{n+1}:\ui\to D_{n+1}$ defined as $Q_{n+1}=j_{n+1}\circ q_{n+1}$, and a retraction $R_{n+1,n}:D_{n+1}\to D_n$. These maps satisfy $H_{n+1}\circ Q_{n+1}=(\gamma_{n+1})^{\light}$, $R_{n+1,n}\circ Q_{n+1}=Q_{n}\circ k_{n+1,n}$, and $R_{n+1,n}\circ i_{n,n+1}=id_{D_n}$.

This recursion results in the following infinite diagram where the top three rows form inverse systems. In the $n$-th column, the vertical composition is $\gamma_n$.
\[\xymatrix{
\ui \ar[r] \ar[d]^-{k_{\infty}} & \cdots \ar[r]^-{id}  & \ui \ar[d]^-{k_4} \ar[r]^-{id} & \ui \ar[d]^-{k_3}  \ar[r]^-{id} & \ui \ar[d]^-{k_2} \ar[r]^-{id} & \ui \ar[d]^-{k_1} \\
\ui \ar[r] \ar[d]^-{Q_{\infty}} & \cdots  \ar[r]^-{K_{5,4}} & \ui \ar[d]^-{Q_4} \ar[r]^-{K_{4,3}} & \ui \ar[d]^-{Q_3} \ar[r]^-{K_{3,2}} & \ui \ar[d]^-{Q_2} \ar[r]^-{K_{2,1}} & \ui \ar[d]^-{Q_1=id_{\ui}} \\
D_{\infty} \ar[r] & \cdots  \ar[r]^-{R_{5,4}} & D_4  \ar[d]^-{H_4} \ar[r]^-{R_{4,3}} & D_3  \ar[d]^-{H_3}  \ar[r]^-{R_{3,2}} & D_2  \ar[d]^-{H_2} \ar[r]^-{R_{2,1}} & D_1=\ui \ar[d]^-{H_1=\gamma_{1}^{\lambda}} \\
& \cdots  & X & X& X& X
}\]
The inverse limit of the top row may be identified with $\ui$ so that the projection maps are also the identity. Since the bonding maps in the second row are non-decreasing continuous surjections, the inverse limit $\varprojlim_{n}(\ui,K_{n+1,n})$ may also be identified with $\ui$ and the bonding maps $K_n:\ui\to\ui$ are also non-decreasing continuous surjections, see \cite[Theorem 4.8 and Lemma 4.2]{Capel}. We let $k_{\infty}=(\varprojlim_{n}k_{n}):\ui\to\ui$ be the inverse limit of the morphisms connecting the first two rows.

In the third row, we have an inverse sequence where the bonding maps $R_{n+1,n}$ are monotone retractions of dendrites. Since any inverse limit of dendrites with monotone bonding maps is a dendrite \cite[Theorem 10.36]{Nadler}, the inverse limit $D_{\infty}=\varprojlim_{n}(D_n,R_{n+1,n})$ is a dendrite and the $n$-th projection $R_n:D_{\infty}\to D_n$ is also a retraction. For $m'\geq m$, let $i_{m,m'}:D_m\to D_{m'}$ and $R_{m',m}:D_{m'}\to D_m$ be the respective composition of the sections $i_{n,n+1}$ and retractions $R_{n+1,n}$ (and the identity if $m=m'$). For fixed $n$, the maps $i_{n,m}$, $m\geq n$ induce a unique map $i_n:D_n\to D_{\infty}$ such that $R_m\circ i_n=i_{n,m}$. The case $m=n$ shows $R_n\circ i_n=id_{D_n}$. Finally, let $Q_{\infty}=\varprojlim_{n}Q_n$ be the inverse limit of the maps connecting the second and third rows. Then the following diagram commutes for all $n\geq 1$.
\[\xymatrix{
\ar@/_3pc/[ddd]_-{\gamma} \ui \ar[d]_-{k_{\infty}} \ar[r]^-{id} & \ui \ar[d]^-{k_{n}} \ar@/^5pc/[ddd]^-{\gamma_n}  \\
\ui \ar[r]^-{K_n} \ar[d]_-{Q_{\infty}} & \ui \ar@/^2pc/[dd]^-{\gamma_{n}^{\light}}\ar[d]^-{Q_n}\\
D_{\infty} \ar@{-->}[d]^-{\exists H_{\infty}}  \ar[r]^-{R_n} & D_n \ar[d]_-{H_n}  \\
X & X
}\]
We include $\gamma$ in the above diagram to indicate that we intend to show that $\gamma$ is constant on the fibers of $Q_{\infty}\circ k_{\infty}$ and therefore induces a unique map $H_{\infty}:D_{\infty}\to X$ such that $H_{\infty}\circ Q_{\infty}\circ k_{\infty}=\gamma$. First, we pause to verify that $k_{\infty}$ and $Q_{\infty}$ are surjective. In the top two rows, $0$ and $1$ are identified in the inverse limit with $(0,0,0,\dots)$ and $(1,1,1,\dots)$ respectively. Since $k_{\infty}$ is continuous and maps $k_{\infty}(0)=(k_n(0))=(0)=0$ and $k_{\infty}(1)=(k_n(1))=(1)=1$, the connectedness of $\ui$ ensures that $k_{\infty}$ is surjective.

To check that $Q_{\infty}$ is surjective, we first show that $i_n(D_n)\subseteq \im(Q_{\infty})$ for all $n\in\bbn$. If $d_n\in D_n$, we have that $i_n(d_n)=(d_1,d_2,d_3,\dots)$ so that $d_k=i_{n,k}(d_n)$ for $k>n$, and $d_k=R_{n,k}(d_n)$ for $k<n$. Fix $t_n\in \ui$ with $Q_n(t_n)=d_n$. For $k<n$, recursively define $t_{k-1}=K_{k,k-1}(t_k)$. Since $Q_{k-1}\circ K_{k,k-1}=R_{k,k-1}\circ Q_k$, it follows that $Q_k(t_k)=d_k$ for all $1\leq k<n$. For $k\geq n$, recursively choose points $t_{k+1}\in q_{k+1}^{-1}(\sigma_{k,k+1}(t_k))$ (see the diagram below to trace these choices).
\[\xymatrix{
&&& \ui \ar[d]_{\sigma_{n,n+1}} \ar[r]^-{Q_n} & D_n \ar[d]^-{i_{n,n+1}} \\
&& \ui \ar@/_1.6pc/[rr]_-{Q_{n+1}} \ar[d]_{\sigma_{n+1,n+2}} \ar[r]^-{q_{n+1}} & D(\gamma_{n+1},\gamma_n) \ar[r]^-{j_{n+1}} & D_{n+1} \ar[d]^-{i_{n+1,n+2}}\\
& \ui \ar@/_2pc/[rrr]_-{Q_{n+2}} \ar[d]_-{\sigma_{n+2,n+3}} \ar[r]^-{q_{n+2}} & D(\gamma_{n+2},\gamma_{n+1}) \ar[rr]_-{j_{n+2}} && D_{n+2} \ar[d]^-{i_{n+2,n+3}}\\
\iddots & \iddots &&& \vdots
}\]
From this choice, we have for every $k\geq n$ that \[Q_{k+1}(t_{k+1})=j_{k+1}\circ q_{k+1}(t_{k+1})=j_{k+1}(\sigma_{k,k+1}(t_k))=i_{k,k+1}\circ Q_{k}(t_k)\] and so, by induction, we have $Q_k(t_k)=d_k$ for all $k\geq n$. It follows that $Q_{\infty}(t_1,t_2,t_3,\dots)=i_n(d_n)$, proving that $i_n(D_n)\subseteq \im(Q_{\infty})$. Now, consider an arbitrary element $x=(d_1,d_2,d_3,\dots)\in D_{\infty}$. We have $x_n=(d_1,\dots,d_{n-1},d_n,i_{n,n+1}(d_n),i_{n,n+2}(d_n),\dots)\in i_n(D_n)$ for all $n\in\bbn$. Since $D_{\infty}$ is topologized as a subspace of $\prod_{n\in\bbn}D_n$, we have $\{x_n\}_{n\in\bbn}\to x$ in $D_{\infty}$. For each $n\in\bbn$, find $t_n\in\ui$ such that $Q_{\infty}(t_n)=x_n$. Since $\ui$ is compact, we may find a subsequence $\{t_{n_m}\}_{m\in\bbn}$ such that $\{t_{n_m}\}_{m\in\bbn}\to t$ for some $t\in\ui$. Since $\{Q_{\infty}(t_{n_m})\}_{m\in\bbn}=\{x_{n_m}\}_{m\in\bbn}\to x$ and $\{Q_{\infty}(t_{n_m})\}_{m\in\bbn}\to Q_{\infty}(t)$, it follows that $Q_{\infty}(t)=x$. Thus $Q_{\infty}$ is surjective.

Knowing that $Q_{\infty}\circ k_{\infty}$ is surjective, we now check that $\gamma$ is constant on each fiber of $Q_{\infty}\circ k_{\infty}$. Suppose $a,b\in\ui$ such that $Q_{\infty}\circ k_{\infty}(a)=Q_{\infty}\circ k_{\infty}(b)$. Then $Q_{n}\circ k_{n}(a)=Q_{n}\circ k_{n}(b)$ for all $n\in\bbn$. Applying $H_n$ gives
\[\gamma_n(a)=H_n\circ Q_{n}\circ k_{n}(a)=H_n\circ Q_{n}\circ k_{n}(b)=\gamma_n(b)\]
for all $n\in\bbn$. Since $\{\gamma_n\}\to\gamma$ uniformly, we have $\{\gamma_n(a)\}_{n\in\bbn}\to\gamma(a)$ and $\{\gamma_n(b)\}_{n\in\bbn}\to\gamma(b)$ but since these sequences in $X$ are equal, it follows that $\gamma(a)=\gamma(b)$. This completes the check and so we conclude that the desired map $H_{\infty}$ exists.

With the existence of $H_{\infty}$ confirmed, we fix $m\in\bbn$ and check that the equality $H_m=H_{\infty}\circ i_m$ holds. Since $\{H_n\circ R_n\circ Q_{\infty}\circ k_{\infty}\}=\{\gamma_n\}\to \gamma=H_{\infty}\circ Q_{\infty}\circ k_{\infty}$ uniformly and $Q_{\infty}\circ k_{\infty}$ is surjective, we have that $\{H_n\circ R_n\}\to H_{\infty}$ uniformly. Recalling that $m$ is fixed, we have $\{H_n\circ R_n\circ i_{m}\}_{n>m}\to H_{\infty}\circ i_m$ uniformly. However, $\{H_n\circ R_n\circ i_{m}\}_{n>m}=\{H_n\circ i_{m,n}\}_{n>m}=\{H_m\}_{n>m}$ is the constant sequence at $H_m$. Thus $H_m=H_{\infty}\circ i_m$ for all $m\in\bbn$.

We will now consider the endpoints of the paths to complete our proof.  For $t\in\{0,1\}$, set $ x_{n,t}=Q_n\circ k_{n}(t)$. Since $R_{n+1,n}(x_{n+1,t})=x_{n,t}$, we have $Q_{\infty}\circ k_{\infty}(t)=(x_{n,t})_{n\in\bbn}$. However, recall that our inductive construction ensures that $i_{n,m}(x_{n,t})=x_{m,t}$ whenever $m>n$. Thus in the limit, we also have $i_{m}(x_{m,t})=(x_{n,t})_{n\in\bbn}$. We conclude that $Q_{\infty}\circ k_{\infty}(t)=i_n\circ Q_n\circ k_{n}(t)$ for all $n\in\bbn$ and $t\in \{0,1\}$. Therefore, if we set $g_{n,1}=Q_{\infty}\circ \ k_{\infty}$ and $g_{n,2}=i_n\circ Q_n\circ k_{n}$, the concatenation $g_{n,1} \ov{g_{n,2}}$ is a well-defined loop in $D_{\infty}$. Moreover, $H_{\infty}\circ g_{n,1}=H_{\infty}\circ Q_{\infty}\circ  k_{\infty}=\gamma$ and $H_{\infty}\circ g_{n,2}=H_{\infty}\circ i_n\circ Q_n\circ k_{n}=H_n\circ Q_n\circ k_{n}=\gamma_n$. Thus $H_{\infty}\circ(g_{n,1}\ov{g_{n,2}})=\gamma\ov{\gamma_{n}}$, proving $\gamma\den\gamma_n$.
\end{proof}

\end{document}